\documentclass{amsart}

\usepackage{latexsym,amssymb}
\usepackage{amscd,amsopn,amsmath,amsthm}
\usepackage{mathrsfs}

\newtheorem{thm}[equation]{Theorem}
\newtheorem{lem}[equation]{Lemma}
\newtheorem{cor}[equation]{Corollary}
\newtheorem{prop}[equation]{Proposition}

\theoremstyle{definition}

\newtheorem{eg}[equation]{Example}

\theoremstyle{remark}
\newtheorem{rmk}[equation]{Remark}

\numberwithin{equation}{section}

\newcommand{\case}[1]{\smallskip\emph{\underline{#1}}:}

\DeclareMathOperator{\Spin}{Spin}           
\DeclareMathOperator{\Sp}{Sp}
\DeclareMathOperator{\PSp}{PSp}
\DeclareMathOperator{\SL}{SL}
\DeclareMathOperator{\GL}{GL}
\DeclareMathOperator{\PGL}{PGL}
\DeclareMathOperator{\PSO}{PSO}
\DeclareMathOperator{\HSpin}{HSpin}

\newcommand{\SO}{\mathrm{SO}}

\DeclareMathOperator{\trdeg}{trdeg}
\DeclareMathOperator{\End}{End}
\DeclareMathOperator{\Id}{Id}
\DeclareMathOperator{\rank}{rank}

\DeclareMathOperator{\car}{char}
\newcommand{\ot}{\otimes}

\DeclareMathOperator{\Lie}{Lie}
\DeclareMathOperator{\Out}{Out}

\newcommand{\C}{\mathbb{C}}
\newcommand{\R}{\mathbb{R}}
\newcommand{\Z}{\mathbb{Z}}
\newcommand{\Q}{\mathbb{Q}}

\newcommand{\G}{\Gamma}

\newcommand{\kx}{k^\times}

\newcommand{\eand}{\quad\text{and}\quad}

\newcommand{\so}{\mathfrak{so}}

\newcommand{\la}{\lambda}
\newcommand{\s}{\sigma}

\newcommand{\cO}{\mathscr{O}}
\newcommand{\E}{\mathscr{E}}
\newcommand{\V}{\mathscr{V}}
\newcommand{\Gc}{\mathscr{G}}
\newcommand{\Rc}{\mathscr{R}}
\newcommand{\Tc}{\mathscr{T}}

\newcommand{\Gct}{\tilde{\mathscr{G}}}
\newcommand{\Tct}{\tilde{\mathscr{T}}}

\newcommand{\qform}[1]{{\left\langle{#1}\right\rangle}}                   

\newcommand{\Mb}{\bar{M}}
\newcommand{\g}{\mathfrak{g}}

\newcommand{\fc}{\mathfrak{c}}
\newcommand{\pgl}{\mathfrak{pgl}}

\newcommand{\kalg}{k_{\mathrm{alg}}}

\newcommand{\Sred}{S\!_{\mathrm{red}}}

\renewcommand{\sl}{\mathfrak{sl}}

\DeclareMathOperator{\im}{im}

\newcommand{\fL}{\mathfrak{L}}

\newcommand{\X}{G}
\newcommand{\Y}{H}
\newcommand{\Yt}{\tilde{H}}

\DeclareMathOperator{\Sym}{Sym}
\DeclareMathOperator{\OO}{O}

\DeclareMathOperator\PP{\mathbb{P}}

\begin{document}

\title{Simple groups stabilizing polynomials}
\author{Skip Garibaldi}
 
\address{Institute for Pure and Applied Mathematics, UCLA, 460 Portola Plaza, Box 957121, Los Angeles, California 90095-7121, USA}
  \email{skip@garibaldibros.com}

\author{Robert M. Guralnick}
 
\address{Department of Mathematics, University of Southern California,
Los Angeles, CA 90089-2532, USA}
\email{guralnic@usc.edu}

\begin{abstract}
We study the problem of determining, for a polynomial function $f$ on a vector space $V$, the linear transformations $g$ of $V$ such that $f \circ g = f$.  In case $f$ is invariant under a simple algebraic group $G$ acting irreducibly on $V$, we note that the subgroup of $\GL(V)$ stabilizing $f$ often has identity component $G$ and we give applications realizing various groups, including the largest exceptional group $E_8$, as automorphism groups of polynomials and algebras.  We show that starting with a simple group $G$ and an irreducible representation $V$, one can almost always find an $f$ whose stabilizer has identity component $G$ and that no such $f$ exists in the short list of excluded cases.  This relies on our core technical result, the enumeration of inclusions $G < H \le \SL(V)$ such that $V/H$ has the same dimension as $V/G$.
The main results of this paper are new even in the special case where $k$ is the complex numbers.
\end{abstract}
 
\maketitle  

\section{Introduction}

The following problem appears in a variety of contexts: Given a polynomial $f$ in $n$ variables, determine the linear transformations $g$ such that $f \circ g = f$.  For the $f$'s studied here it is obvious that such a $g$ must be invertible, so the answer will be a subgroup of $\GL_n$.  Frobenius's 1897 paper \cite{Frobenius} and Dieudonne's 1943 paper \cite{Dieu:LPP} are both aimed at solving the special case where $f$ is the determinant.  Solutions to this problems appear in many places in algebra as well as in geometric complexity theory, see for example \cite{MuS}.

This problem is typically solved using arguments that are special to the particular polynomial $f$ being studied.  Here we show that a single result gives the answer for a large class of $f$'s.  Roughly speaking, if $f$ is defined on a vector space $V$ and is invariant under the action of a simple algebraic group $G$ that acts irreducibly on $V$, we show that ``typically'' the stabilizer $O(f)$ of $f$ has identity component $G$.  With this in hand, it is not hard to determine the full group $O(f)$.

This can be viewed as a sort of reverse invariant theory.  Suppose an algebraic group $G$ acts on a vector space $V$ and you pick a $G$-invariant polynomial $f$ on $V$.  The stabilizer $O(f)$ contains $G$ but a priori might be bigger.  It is known, for example, that in case $G$ is a semisimple adjoint complex Lie group and $V = \Lie(G)$, then $G$ is the identity component of $\cap_{f \in \C[V]^G} O(f)$, see \cite{Dixmier}.  In contrast, we show in \S\ref{canonical.sec} below that for simple $G$ apart from $\SO_5$, not only is no intersection necessary, but a single homogeneous generator $f$ of $\C[V]^G$ will do.  As further illustrative examples, we show (1) that $E_8$ is the (identity component of) the automorphism group of an octic form in 248 variables and of a 3875-dimensional algebra, see sections \ref{E8.sec}, \ref{E8.adj.sec}, and \ref{E8.3875.sec}; and (2) that, up to isogeny and excluding fields of small characteristic, every simple group is the identity component of the stabilizer of a \emph{cubic} form, see \S\ref{cubic.sec}.  This latter example shows that the degree of a homogeneous $f$ need not give any information about the identity component of $O(f)$.

The core idea in this paper is that there cannot be many closed connected overgroups $H$ such that $G < H \le \SL(V)$ and that furthermore there are extremely few such $H$ such that $V/H$ and $V/G$ have the same dimension, equivalently, the field $k(V)^G$ is an algebraic extension of $k(V)^H$.\footnote{In contrast, for finite $G$ and $H$, every inclusion $G < H < \SL(V)$ leads to a proper algebraic extension $k(V)^G \supsetneq k(V)^H$ by the fundamental theorem of Galois theory.} Indeed, there are so few such $H$ that we can enumerate them in Theorem \ref{distinctinvariants}.  Because of this, we can show that for most pairs $(G, V)$, there is a polynomial $f$ whose stabilizer has identity component $G$ and that in the excluded cases there is no such $f$.

We work with both groups (in the naive sense) as well as affine group schemes over an arbitrary field $k$.  Our Theorems \ref{phys}, \ref{e8}, \ref{canonical}, \ref{for:cubic0}, \ref{cubic}, \ref{distinctinvariants}, and \ref{MT} are new already in the case where $k$ is the complex numbers, and Theorem \ref{generic} is an analogue for $k$ an algebraically closed field of prime characteristic that was previously known only in characteristic zero.


\subsection*{Notation and conventions}
An \emph{algebraic group scheme} over a field $k$ is an affine group scheme of finite type over $k$ as defined in \cite{SGA3:new}.  An \emph{algebraic group} is an algebraic group scheme that is smooth.  For an algebraic group scheme $G$ over $k$ and an extension $k'$ of $k$, we put $G(k')$ for the group of $k'$-points of $G$, i.e., the $k$-algebra homomorphisms $k[G] \to k'$; it is a (concrete) group.  If $G$ is smooth and (a) $k$ is algebraically closed or (b) $k$ is infinite and $G$ is reductive, then $G(k)$ is Zariski-dense in $G$ and sometimes in these cases we will conflate $G$ and $G(k)$ as is commonly done.

For a finite-dimensional vector space $V$ over $k$, we write $k[V]$ for the ring of polynomial functions on $V$ with coefficients in $k$, i.e., for the symmetric algebra on the dual space of $V$.  The \emph{(naive) stabilizer} in $\GL(V)$ of an $f \in k[V]$ is the (concrete) subgroup $\{ g \in \GL(V) \mid \text{$f \circ g = f$ in $k[V]$} \}$.  The \emph{scheme-theoretic stabilizer} of $f$ is the sub-group-scheme of $\GL(V)$ centralizing $f$ in the sense of \cite[A.1.9]{CGP}.  In ``most" cases, such as if  $\car k$ is zero or larger than some bound depending on $V$ and $f$, the scheme-theoretic stabilizer will be smooth; if additionally (a) or (b) from the previous paragraph hold, then the two notions of stabilizer coincide.

The rational irreducible representations (the \emph{irreps}) of a simple algebraic group $G$ are denoted $L(\la)$ where $\la$ is a dominant weight for $G$.  (We only consider rational representations in this paper.)  Each $\la$ can be written uniquely as a sum of fundamental dominant weights $\la_1, \ldots, \la_r$ of $G$, which we number as in \cite[Chap.~VI, Plates I--IX]{Bou:Lie}.  If $k$ has prime characteristic $p$, the \emph{restricted} representations are those $L(\sum c_i \la_i)$ such that $0 \le c_i < p$ for all $i$, and every irreducible representation can be expressed uniquely as a tensor product of Frobenius twists of restricted ones.  If $k$ has characteristic 0, then every irrep is restricted, by definition.

\section{Reminders on group actions}  \label{remind.sec}

Suppose $G$ is a connected algebraic subgroup of $\GL(V)$.  For each $v \in V$, the dimension of the $G$-orbit $Gv$ and the stabilizer $G_v$ of $v$ are related by the equation $\dim Gv + \dim G_v = \dim G$, as follows by applying the fiber dimension theorem \cite[\S13.3]{EGA4.3} to the map $G \to V$ defined via $g \mapsto gv$.

We define $k[V]^G$ to be the subring of $k[V]$ consisting of elements $f$ that are sent to $f \ot 1$ under the comodule map $k[V] \to k[V] \otimes k[G]$.  If $G$ is reductive and $k$ is infinite, then $G(k)$ is dense in $G$ and 
\[
k[V]^G = \{ f \in k[V] \mid \text{$f \circ g = f$ in $k[V]$ for all $g \in G(k)$} \},
\]
i.e., the collection of $f \in k[V]$ whose naive stabilizers contain $G(k)$.

Put $\kalg$ for an algebraic closure of $k$.  As in \cite[Th.~2]{Rosenlicht:basic}, \cite{Sesh:GR}, or \cite[\S2]{PoV}, there is a nonempty and $G$-invariant open subset $U$ of $V \ot \kalg$ such that 
\[
   \trdeg_{/k} k(V)^G = \dim V - \dim Gv   \quad \text{for $v \in U$}.
\]
On the other hand, if $G$ is semisimple, an easy argument as in \cite[Th.~3.3]{PoV} shows that $k(V)^G$ is the fraction field of $k[V]^G$ and we have
\begin{equation} \label{dim.eq}
\dim k[V]^G = \dim V - \dim G + \dim G_v \quad \text{for $v \in U$,}
\end{equation}
i.e., the Krull dimension of $k[V]^G$ equals the codimension of a generic orbit in $V$.  

We remark that the orbits in $U$ are orbits of maximal dimension in $V \ot \kalg$, as can be seen by applying upper semicontinuity as in \cite[Proposition VI.4.1(i)]{SGA3:new} to the collection of stabilizers, which form a group scheme over $V$.  Furthermore, if $\car k = 0$, the conjugacy class of the stabilizer of $u \in U$ does not depend on the choice of $u$, as follows from the Luna stratification \cite[Th.~7.2]{PoV}.
 Regardless of the characteristic, we know the following.
\begin{lem} \label{stab.const}
Let $G$ be a connected algebraic group acting on an irreducible variety $V$ over an algebraically closed field $k$ such that $\dim G_v = 0$ for some $v \in V$.  Then there exists a nonempty open subvariety $U$ of $V$ such that $|G_v(k)|$ is finite and constant for $v \in U$.
\end{lem}

\begin{proof}
Take $Y$ to be the closure of the image of the map $f \!: G \times V \to V \times V$ with $f(g,v) = (v,gv)$.  Note that the fiber over $(v,gv)$ has $k$-points $\{(h, v) \mid hv=gv \}$
and so has cardinality equal to $|G_v(k)|$.  The hypothesis that some $G_v$ is zero-dimensional gives the same conclusion for generic $v \in V$, whereby the map $f \!: G \times V \to \overline{\im f}$ is generically finite.  It follows (by, for example, \cite[Th.~5.1.6(iii)]{Sp:LAG}) that there is a nonempty open subvariety $Y$ of $\overline{\im f}$ such that all fibers have the same size, equal to the separable degree of the finite extension of function fields $k(G \times V)/k(Y)$.  The projection of $Y$ into the first copy of $V$ contains an
open subvariety of $V$ (because the image of the morphism projects onto the first copy of $V$), hence the claim.
\end{proof}

\subsection*{Comparing invariants under $G$ and $\Lie(G)$} 
For $f \in k[V]$, we can adjoin an indeterminate $t$ to $k$ and expand, for $v, v' \in V$:
\[
f(v + tv') = f(v) + t f_1(v, v') + (\text{terms of higher degree in $t$}),
\]
where $f_1(v, v')$ is the directional derivative of $f$ at $v$ in the direction $v'$.  We say that $f$ is \emph{Lie invariant} under $X \in \End(V)$ if $f_1(v, Xv) = 0$ for all $v \in V$.  

\begin{eg} \label{lie.scalar}
If $f$ is homogeneous, then $f_1(v,v) = (\deg f) \, f(v)$ for all $v \in V$, because it is true for every monomial.  Thus: $f$ is Lie invariant under the scalar matrices iff $\car k$ divides $\deg f$.
\end{eg}

For an affine group scheme $G \le \GL(V)$, we can view $\Lie(G)$ as the fiber over $1_G$ of the map $G(k[x]/(x^2)) \to G(k)$ induced by $x \mapsto 0$.  From this, it is obvious: If $f \in k[V]$ is invariant under the group $G$, then $f$ is Lie invariant under  $\Lie(G)$.
The converse holds if $\car k = 0$ \cite[Lemma 2]{Jac:J1}, but in characteristic $p \ne 0$ the situation is more complicated.  For example, every element of the subalgebra $k[V]^{(p)}$  generated by  $\{ h^p \mid h \in k[V] \}$ is Lie invariant under $\Lie(G)$.  We have the following, which is an application of a result of Skryabin:

\begin{lem} \label{lie.inv}
Let $V$ be a representation of a semisimple algebraic group $H$ over an algebraically closed field $k$ of prime characteristic $p$.
\begin{enumerate}
\setcounter{enumi}{-1}
\item \label{lie.0} If $k[V]^H = k$, then the subring of $k[V]$ of elements Lie invariant under $\Lie(H)$ is $k[V]^{(p)}$.
\item \label{lie.1} If $k[V]^H = k[f]$ for some homogeneous $f$ of degree not divisible by $p$, then the subring of $k[V]$ of elements Lie invariant under $\Lie(H)$ is $k[V]^{(p)}[f]$, a free $k[V]^{(p)}$ module of rank $p$.
\end{enumerate}
\end{lem}

Recall that case \eqref{lie.0} of the lemma encompasses all representations with $\dim k[V]^H$ equal to 0, and that every representation with $\dim k[V]^H = 1$ has $k[V]^H = k[f]$ for some homogeneous $f$, see \cite[Prop.~12]{Popov:14} for $k = \C$ and \cite[6.1]{BGL} for the general case

\begin{proof}
For \eqref{lie.0}, this is a straightforward application of \cite[Th.~5.5]{Skryabin}.  For \eqref{lie.1}, Example \ref{lie.scalar} shows that, for each $v \in V$, $f_1(v,v) = 0$ if and only if $f(v) = 0$.  Therefore, the variety $Z$ consisting of those $v \in V$ such that the linear form $f_1(v,-)$ vanishes is contained in the vanishing set $Y$ of $f$, and in fact is the singular set of $Y$. Since $f$ is irreducible (because $H$ is connected and has only trivial characters), $Z$ is a proper subvariety, 
so $Z$ has codimension at least 2 in $V$ and Skryabin's result gives the claim.
\end{proof}
\section{The compact real form of $E_8$} \label{E8.sec}

More than 125 years ago, Wilhelm Killing classified the finite-dimensional simple complex Lie algebras by introducing the notion of root system and then classifying the simple root systems.\footnote{Apart from some small errors, corrected in \cite{Ca:th}.}  In the paper containing the classification, \cite{Killing2}, he explicitly posed the opposite problem of giving, for each simple root system, a concrete description of a simple Lie algebra with that root system (ibid., p.~38).  He answered this problem for the root systems of types $A_n$, $B_n$, $C_n$, and $D_n$ by showing that they come from $\sl_{n+1}$, $\so_{2n+1}$, $\mathfrak{sp}_{2n}$, and $\so_{2n}$ respectively, and these descriptions are now a standard part of the theory as in \cite[\S{VIII.13}]{Bou:Lie} or \cite[7.4.7]{Sp:LAG}.  Analogous descriptions for types $E_6$ and $E_7$ date back 120 years to Cartan's thesis \cite[pp.~139--147]{Ca:th} and were followed by treatments of $G_2$ by Engel \cite{Engel} and Cartan (without proof, \cite[p.~298]{Ca:real}) and $F_4$ by Chevalley--Schafer \cite{ChevS} and a refinement of the $E_7$ case by Freudenthal \cite{Frd:E7}.  For $E_8$, the only such interpretation known is as the automorphism group of its Lie algebra\footnote{See \S\ref{arbfield.sec} for a more comprehensive discussion.}, and we now give another one.
Recall that the smallest  faithful  irreducible representation of a group of type $E_8$ is its adjoint representation of dimension 248.

The paper \cite{CederwallPalmkvist} gives an explicit formula for a degree 8 invariant polynomial $f$ on the Lie algebra of the compact real form of $E_8$ that is not in the span of the 4th power of the Killing form $q$, obtained by decomposing the representation with respect to the maximal subgroup of type $D_8$.  Alternatively, such an invariant may be found by picking any degree 8 polynomial  $f_0$ and defining $f$ to be the average with respect to Haar measure, $f(v) := \int_G f(gv)\, dg$; the resulting $f$ will be $E_8$-invariant and almost certainly not in $\R q^4$.

\begin{thm} \label{phys}
The stabilizer in $\GL_{248}(\R)$ of the octic polynomial $f$ displayed in \cite[(2.3)]{CederwallPalmkvist} is generated by the compact form of $E_8$ and $\pm 1$.  The stabilizer of $f$ in $\GL_{248}(\C)$ is generated by the complex $E_8$ and the eighth roots of unity.
\end{thm}

\begin{proof} 
As the compact real $E_8$ and the eighth roots of unity stabilize $f$ by construction, it suffices to verify that nothing else stabilizes $f$, for which it suffices to consider the complex case.
Put $S$ for the identity component of the stabilizer of $f$ in $\GL_{248}$.  Because the representation is irreducible, it follows that $S$ is reductive, hence semisimple because $f$ is not constant, hence simple because the representation is tensor indecomposable for $E_8$.  If $S$ properly contains $E_8$, then it has classical type and its smallest nontrivial representation has dimension 248 and is not symplectic, i.e., $S$ is $\SL_{248}$ or $\SO_{248}$.   But $\SL_{248}$ does not stabilize any nonzero octic form and the only octic forms left invariant by $\SO_{248}$ are scalar multiples of the fourth power of  a quadratic form (the Killing quadratic form for $E_8$), so we conclude that $S = E_8$.  As the full stabilizer of $f$ normalizes $S$ and $E_8$ has no outer automorphisms, the full stabilizer is contained in the group generated by $E_8$ and the scalar matrices; the claim follows.
\end{proof}

The compact real form of $E_8$ discussed in the theorem is the one playing a role in the recent laboratory experiment described in \cite{Coldea}, cf.~\cite{BoG}.

We will generalize Theorem \ref{phys} to other fields in Theorem \ref{e8} and to other groups in Theorem \ref{canonical}.  Nonetheless, we have included this doubly special result here for two reasons.  First, it is an example where the polynomial function is known explicitly.  Second, despite that it is a very special case of our results below, it already sheds new light on the problem posed by Killing more than 125 years ago.

\section{Containment of Lie algebras} \label{lie.sec}

We will prove the following result, which will be used to prove that certain group schemes are smooth, see Theorem \ref{e8} and Theorem \ref{for:cubic0}.

\begin{prop} \label{containment} 
Let $G \le \SL(V)$ be a simple algebraic group over an algebraically closed field $k$ of prime characteristic such that $V$ is irreducible and tensor indecomposable.  
Suppose that $\Lie(G) < \fL \le \sl(V)$ are containments of restricted Lie algebras such that $\fL$ is invariant under $G$, and $\fL$ is minimal with this property.  If $\car k \ne 2, 3$, $\Lie(G)$ is simple, and the centralizer of $\Lie(G)$ in $\fL$ is $0$, then there exists a simple simply connected algebraic group $H$ and a homomorphism $\phi \!: H \to \SL(V)$  so that $G < \phi(H)$ and $\fL = d\phi(\Lie(H))$.
\end{prop}

\begin{proof}
Write $\g$ for  $\Lie(G)$.
Let $I$ be the subalgebra of $\fL$ generated by those $c \in \fL$ such that $[c[c\fL]] = 0$, and suppose $I \ne 0$.  It is $G$-invariant because $G$ acts by algebra automorphisms on $\fL$, hence $[\g, I] \le I$, and by minimality $\fL = I + \g$.  Since $I$ is nilpotent, by minimality it is abelian.  Moreover, it acts completely reducibly on $V$ because the socle is $G$-invariant, thus $I$ is a toral subalgebra.  As $\g$ does not centralize $I$, $G$ acts nontrivially on $I$, an impossibility, so $I = 0$.

Suppose now that $[\fL \fL]$ is smaller than $\fL$.  As $\g$ is perfect, minimality of $\fL$ implies that $[\fL\fL] = \g$.  But $\fL$ normalizes $[\fL\fL]$, and every derivation of $\g$ is inner \cite{Rudakov}, so $\fL/\g$ is naturally identified with the centralizer of $\g$ in $\fL$, i.e., 0.  This contradicts the hypothesis that $\fL \ne \g$, so $[\fL\fL] = \fL$.

Now \cite[Th.~3]{Premet:strong} (for $\car k > 5$) and \cite{Premet:strong5} (for $\car k = 5$) give that $\fL$ is a sum of simple ideals ``of classical type".  Since $V$ is tensor indecomposable for $G$ and $V$ is restricted, it is also tensor indecomposable for $\g$.  It follows that $\fL$ is itself simple of classical type, cf.~\cite[Lemma 3.1]{BlockZ}.  Steinberg \cite{St:rep} gives a simple simply connected group $H$ and a homomorphism $\phi \!: H \to \SL(V)$ such that $\fL$ is the subalgebra of $\sl(V)$ generated by the images of the root subalgebras of $\Lie(H)$ under $d\phi$.
\end{proof}

\section{Adjoint representation of $E_8$} \label{E8.adj.sec}
The proof of Theorem \ref{phys} essentially relied on the nonexistence of overgroups of $E_8(\C)$ in $\SL_{248}(\C)$.  
This can be generalized as follows, which exploits the observation that overgroups of simple groups in irreducible representations are comparatively rare.  In the statement, $k$ has characteristic $p \ge 0$; in case $p = 0$, we set $k[V]^{(p)} = k$.

\begin{lem} \label{slem}
Let $X < \SL(V)$ be a simple algebraic group over an algebraically closed field $k$ such that $V$ is irreducible, restricted, and tensor indecomposable\footnote{If $\car k = 0$, then irreducible implies restricted and tensor indecomposable.  If $\car k \ne 2, 3$, then irreducible and restricted implies tensor indecomposable \cite[1.6]{seitzmem}.} for $X$.  Put $q$ for a nonzero $X$-invariant quadratic form on $V$ if one exists; otherwise set $q := 0$.
If $(X, V)$ does not appear in Table 1 of \cite{seitzmem}, then for every $f \in k[V]^X \setminus k[q]$,
the stabilizer of $f$ in $\GL(V)$ has identity component $X$.
If additionally $\car k \ne 2, 3$ and does not divide $\deg f$,  $f$ is not in $k[V]^{(p)}[q]$, and furthermore $\car k$ does not divide $n+1$ if $X$ has type $A_n$, then the scheme-theoretic stabilizer of $f$ in $\GL(V)$ is smooth with identity component $X$.
\end{lem}

\begin{proof}
Put $S$ for the identity component of the scheme-theoretic stabilizer of $f$ in $\GL(V)$, and $\Sred$ for its reduced subgroup.  Because $\Sred$ contains $X$ and $V$ is an irreducible representation of $X$, it follows that $\Sred$  is reductive, hence semisimple because $f$ is not constant.
If $X \ne \Sred$, then as $(X,V)$ is not contained in Table 1 of \cite{seitzmem}, p.278 of ibid.\ gives that $\Sred$ is the stabilizer of a symplectic (in case $\car k = 2$) or quadratic form on $V$, hence $k[V]^{\Sred} = k[q]$, contradicting the existence of $f$.  Therefore, $X = \Sred$ and the first claim follows.

For the second claim, we have natural containments $\Lie(X) \le \Lie(S) \le \sl(V)$ by \cite[\S{VII$_{\mathrm{A}}$.6}]{SGA3:new}.  The hypothesis on the characteristic guarantees that $\Lie(X)$ is a simple Lie algebra \cite[2.7a]{Hogeweij}.  Since $V$ is a restricted irrep of $X$, it is also an irrep of $\Lie(X)$, hence the centralizer of $\Lie(X)$ in $\Lie(S)$ consists of scalar matrices, so is 0 by Example \ref{lie.scalar}. Proposition \ref{containment} provides a simple, simply connected algebraic group (scheme) $H$ and a homomorphism $\phi \!: H \to \SL(V)$ with $\Lie(S)$ containing $d\phi(\Lie(H))$.  

The image $\phi(H(k))$ of the abstract group of $k$-points of $H$ is a subgroup of $\SL(V)(k)$ containing $X(k)$, so by Seitz it is $\SL(V)(k)$, $\SO(V)(k)$, or $\Sp(V)(k)$.  The map $\phi$ is a central isogeny by construction \cite[2.15]{BoTi:C}, and combining this with \cite[(A)]{BoTi:hom} gives that $H$ is isomorphic to $\SL(V)$, $\Spin(V)$, or $\Sp(V)$, respectively.  Examining the list of dimensions of the irreps of $H$ from \cite{luebeck}, we see that $\phi$ is equivalent to $V$ or its dual, hence $\ker \phi$ is zero and the subalgebra $d\phi(\Lie(H))$ of $\Lie(S)$ is $\sl(V)$, $\so(V)$, or $\mathfrak{sp}(V)$.  But $f$ is Lie invariant under $\Lie(S)$, contradicting Lemma \ref{lie.inv}, so $S$ is smooth.  
\end{proof}

\begin{rmk} \label{homogeneous}
Suppose $f$ is non-constant homogeneous and $k$ is algebraically closed.
Many of our results show that the naive stabilizer of $f$ in $\GL(V)$ has identity component a simple group $G$, in which case the naive stabilizer of $kf$ in $\PGL(V)$ will have identity component the image of $G$.  In Lemma \ref{slem}, Theorem \ref{e8}, and Theorem \ref{for:cubic0}, we also prove that the scheme-theoretic stabilizer of $f$ is smooth, in which case its image --- the scheme-theoretic stabilizer of $kf$ in $\PGL(V)$ --- is also smooth. 
\end{rmk}

``Most'' pairs $(X, V)$ with $V$ irreducible and tensor indecomposable satisfy the hypotheses of Lemma \ref{slem} above.  Indeed, unless $(X, V)$ appears in (the rather short) Tables \ref{dim0} or \ref{dim1}, then there exists an $f \in k[V]^X \setminus k[q]$, see \S\ref{dim.sec}.  Furthermore,
inspecting the table in Seitz shows that when $\car k = 0$, a randomly selected irreducible representation $V$ of any particular simple $X$ will satisfy the hypotheses of Lemma \ref{slem} with probability 1, when one defines this probability as a limit over finite sets of weights of increasing size.

We also use Lemma \ref{slem} to give a  version of Theorem \ref{phys} for any field.  Because of the importance of this one example, we give a quick proof of Seitz's result for this case.

\begin{lem} \label{seitz4e8}   Let $k$ be an algebraically closed field of characteristic $p \ge 0$.
Let $G=E_8(k) < \SL_{248}(k)$.  
If $p \ne 2$, there is a unique proper closed subgroup $H$ with $G < H < \SL_{248}(k)$, and
it is isomorphic to $\SO_{248}(k)$.  
If $p =2$, there are  precisely $3$ proper closed subgroups $H_i, 1 \le i \le 3$ of
$\SL_{248}(k)$ properly containing $G$.  We have
$G < H_1 < H_2 < H_3 < \SL_{248}(k)$ with $H_1 \cong \SO_{248}(k)$, 
$H_2 \cong \OO_{248}(k)$ and $H_3 \cong \Sp_{248}(k)$. 
\end{lem}

\begin{proof}  It suffices to consider closed connected subgroups.   Since $G$ acts tensor 
indecomposably on the adjoint module, any connected overgroup $H$ of $G$ is simple.
Since the rank of $H$ is greater than $8$, $H$ must be of type  $A$, $B$, $C$, or $D$.   Moreover,
as any representation of $E_8$ has dimension at least $248$, the same must be true
of $H$.    Thus $H  \cong \SL_{248}$, $\SO_{248}$ or $\Sp_{248}$.   If $p \ne 2$,  $G$
does not preserve an alternating form.  In any case, since $G$ acts irreducibly, $G$
preserves a unique (up to scalar) quadratic form (or symplectic form if $p=2$).
The result follows.
\end{proof}  

\begin{thm} \label{e8}
Let $G$ be a simple algebraic group of type $E_8$ over a field $k$ and put $q$ for a nonzero $G$-invariant quadratic form on  $V := \Lie(G)$.  Then there exists a homogeneous polynomial $f$ of degree $8$ on $V$ that is $G$-invariant and does not belong to $kq^4$.  For each such $f$,
\begin{enumerate}
\item \label{e8.naive} the stabilizer of $f$ in $\GL(V)$ is generated by $G(k)$ and the eighth roots of unity in $k$; and
\item \label{e8.smooth} if $\car k \ne 2, 3$, the scheme-theoretic stabilizer of $f$ in $\GL(V)$ and the scheme-theoretic stabilizer of $kf$ in $\PGL(V)$ is  (the image of) $G$.
\end{enumerate}
\end{thm}

\begin{proof} 
Suppose first that $k$ is algebraically closed.  Put $\E_8$ for a split group scheme of type $E_8$ over $\Z$ (and identify $\E_8(k)$ with $G$) and $\mathfrak{q}$ for an indivisible $\E_8$-invariant quadratic form on $\Lie(\E_8)$.  As the space of octic $\E_8(\C)$-invariant polynomials on $\Lie(\E_8)\ot \C$ is 2-dimensional, the rank of the corresponding module over $\Z$ is 2.  This dimension can only increase when we reduce modulo the characteristic of $k$, so there is an octic $\E_8$-invariant polynomial $f$ on $\Lie(\E_8) \ot k$ that is not a multiple of $(\mathfrak{q} \ot k)^4$, equivalently, is not a multiple of $q^4$. 
The stabilizer of $f$ in $\GL(V)$ is generated by $\E_8(k)$ and the eighth roots of unity by Lemma \ref{slem} or Lemma \ref{seitz4e8}.  Claim \eqref{e8.smooth} is a direct application of Lemma \ref{slem} and Remark \ref{homogeneous}.

Now let $k$ be arbitrary.  The natural homomorphism $k[V]^G \ot \kalg \to \kalg[V \ot \kalg]^{G \times \kalg}$ is an isomorphism \cite[Lemma 2]{Sesh:GR}, so there exists an $f \in k[V]^G \setminus k[q]$.  As $G(k) = G(\kalg) \cap \GL(V)$, claim \eqref{e8.naive} follows.  Claim \eqref{e8.smooth} is obvious because it can be verified after base change to $\kalg$.
\end{proof}

We conjecture that the scheme-theoretic stabilizer of $kf$ in $\PGL(V)$ is smooth for all $k$, and that the scheme-theoretic stabilizer $S$ of $f$ is also smooth when $\car k = 3$.
However, if $\car k = 2$, then $S$ is not smooth, because $\Lie(S)$ contains both $\Lie(G)$ and the scalar matrices (because they are the Lie algebra of the group scheme of eighth roots of unity), so $\dim \Lie(S) > \dim \Lie(G) = \dim G = \dim S$.

\begin{rmk}
The method of proof used in this section can be applied more generally to argue for example that $G$ is the identity component of the stabilizer of a subset of $V^{\otimes r} \otimes (V^*)^{\otimes s}$ for some $r$, $s$.  As a concrete illustration, consider $V$ the minuscule 56-dimensional representation of a group $G$ of type $E_7$ over an algebraically closed field $k$ of characteristic 2.  Then $G$ stabilizes a nonzero quadratic form $q$, and $k[V]^G = k[q]$ (see Prop.~\ref{few}), so $G$ is not the identity component of the stabilizer of a homogenous form.  But there is a non-symmetric 4-linear form $\Psi$ on $V$ whose stabilizer has identity component $G$, see \cite[\S6]{Lurie} and \cite{Luzgarev}.  This claim could be proved following the methods of this section by checking that $\SO(q)$ does not stabilize $\Psi$.

(The case where $k$ has characteristic different from 2 is easier.  Then $k[V]^G = k[f]$ for a quartic form $f$ and Lemma \ref{slem} says that the stabilizer of $f$ has identity component $G$.  So in any characteristic  $G$ is the identity component of the stabilizer of a degree 4 element in the tensor algebra on $V$.)
\end{rmk}

\section{Adjoint groups are stabilizers of canonical homogeneous forms} \label{canonical.sec}

In this section, we show that each split adjoint group, roughly speaking, can be realized as the identity component of the stabilizer of a \emph{canonical} homogeneous form on its Lie algebra.

To see this, fix a simple root system $\Rc$ and put $A$ for the ring obtained by adjoining to $\Z$ the inverses of the torsion primes listed in Table \ref{badprimes}, and also adjoining $1/2$ if $\Rc$ has type $C_\ell$ for some $\ell \ge 1$. (This list is chosen in order to apply the results of \cite{Dem:inv}.) This data uniquely determines a split adjoint algebraic group $\Gc$ over $A$ of type $\Rc$ \cite[\S{XXV.1}]{SGA3:new}.
Let $\Tc$ be a split maximal $A$-torus in the simply connected cover $\Gct$ of $\Gc$ and put $W$ for the Weyl group $N_{\Gct}(\Tc)/\Tc$.
\begin{table}[hbt]
\begin{center}
\begin{tabular}{c|ccccccccc}
&&$B_\ell$ ($\ell \ge 3$) or \\
type of $\Rc$&$A_\ell$&$D_\ell$ ($\ell \ge 4$)&$C_\ell$&$E_6, E_7, F_4$&$E_8$&$G_2$\\ \hline
torsion primes&$\emptyset$&2&$\emptyset$&2, 3&2, 3, 5&2 \\ \hline
not very good&divisors&2&2&2, 3&2, 3, 5&2, 3 \\
primes& of $\ell + 1$
\end{tabular}
\caption{Torsion primes and primes that are not very good} \label{badprimes}
\end{center}
\end{table}

The free module $\Lie(\Tc)$ is naturally identified (via a pinning) with $Q^\vee \ot A$ where $Q^\vee$ denotes the root lattice of the dual root system $\Rc^\vee$.
It is classical that $\R[Q^\vee \ot \R]^W$ is a polynomial ring with homogeneous generators $p_1, p_2, \ldots, p_\ell$ where 
\begin{equation} \label{p.degs}
2 = \deg p_1 < \deg p_2 \le \cdots \le \deg p_{\ell-1} < \deg p_\ell = (\text{Coxeter \# of $\Rc$}),
\end{equation}
and that these degrees are all distinct unless $\Rc$ has type $D_{2\ell'}$ in which case both $p_{\ell'}$ and $p_{\ell' + 1}$ have degree $\ell'$.  These generators are not uniquely determined.

\begin{eg}[flat bases] \label{flat}
Whatever the type of $\Rc$,
one can impose an additional condition on the generators of the real ring of invariants: Write $\qform{\, , \, }$ for a Weyl-invariant symmetric bilinear form on $P \ot \R$ where $P = (Q^\vee)^*$ is the weight lattice with basis the fundamental dominant weights $\omega_1, \ldots, \omega_\ell$ and define a bilinear map $b$ on polynomials in $P \ot \R$, $b \!: \R[Q^\vee \ot \R] \times \R[Q^\vee \ot \R] \to \R[Q^\vee \ot \R]$ by 
\[
b(p, p') = \sum_i \sum_j \frac{\partial p}{\partial \omega_i} \frac{\partial p'}{\partial \omega_j} \qform{\omega_i, \omega_j}.
\] 
In \cite{SaitoYS}, the generators $p_1, \ldots, p_\ell$ are said to be a \emph{flat basis} if $\frac{\partial}{\partial p_\ell}b(p, p')$ belongs to $\R$ for all $p, p'$.  (This definition was motivated by the study of logarithmic poles, see \cite{Saito:linear}.)  Flat bases were constructed in \cite{Talamini} for types $E_7$ and $E_8$ and in \cite{SaitoYS} for the remaining types.  The latter paper also proved that there is a \emph{unique} flat basis up to scaling the elements by nonzero real numbers, or interchanging the two invariants of degree $\ell'$ in case $\Rc$ has type $D_{2\ell'}$.
\end{eg}

\begin{lem} \label{demlem}
$A[\Lie(\Tc)]^W$ is a  polynomial ring with homogeneous and indivisible generators $p_1, \ldots, p_\ell$ with degrees as in \eqref{p.degs}.  For every homomorphism of $A$ into a field $k$, the natural map $A[\Lie(\Tc)]^W \ot k \to k[\Lie(\Tc) \ot k]^W$ is an isomorphism and $k[\Lie(\Tc) \ot k]^W$ is a polynomial ring with generators the images of $p_1, \ldots, p_\ell$.
\end{lem}

\begin{proof}
The main result of \cite{Dem:inv} says that the arrow in the statement of the lemma is an isomorphism and that the rings are graded polynomial rings.  Taking $k = \R$ and tracing through the proof of ibid., Lemma 6, shows that $A[\Lie(\Tc)]^W$ has homogeneous indivisible generators of the same degrees as those of $\R[\Lie(\Tc) \ot \R]^W$.
 \end{proof}
 
In view of Lemma \ref{demlem}, we may:
 \begin{equation}\label{canonical.def}
\parbox{3.8in}{\emph{Choose indivisible homogeneous $p_1, \ldots, p_\ell \in A[\Lie(\Tc)]^W$ whose images in $\R[\Lie(\Tc) \ot \R]^W$ are a generating set.}}
\end{equation}

 \begin{eg}[type $A_\ell$]
For $\Rc$ of type $A_\ell$, we may identify $\Lie(\Tc)$ with the space of $(\ell+1)$-vectors whose coordinates sum to zero, which identifies $A[\Lie(\Tc)]^W$ with $A[x_1, \ldots, x_{\ell+1}]/(\sum x_i)$.  The fundamental theorem of symmetric polynomials gives that one may take $p_i$ to be the elementary symmetric polynomial in $x_1, \ldots, x_{\ell + 1}$ of degree $i+1$.

The generators $p_1, p_2$ of degrees 2, 3 respectively are the same as those in the flat basis in Example \ref{flat}, but $p_i$ for $i \ge 3$ are different in the two cases, as follows from \cite[2.5.4]{SaitoYS}.
\end{eg}

Fix a homomorphism of $A$ to a field $k$.  The natural map $k[\Lie(\Gct) \ot k]^{\Gc} \to k[\Lie(\Tc) \ot k]^W$ is an isomorphism by \cite[\S{II.3}]{SpSt} or \cite[Th.~4(i)]{KW}, and we write $f_i$ for the pullback of the element $p_i$ chosen in \eqref{canonical.def}.  Note that $f_1$ is a $\Gc$-invariant quadratic form, a scalar multiple of the Killing form.

If $\Rc$ has type $A_1$, then $\Gc$ is $\SO_3$, the identity component of the stabilizer of $f_1$.  For $\Rc$ of higher rank, we have the following:

\begin{thm} \label{canonical}
Suppose $\Rc$ is not of type $A_1$ nor $C_2$, $\car k$ is very good for $\Rc$, and $f \in \{ f_2, \ldots, f_\ell \}$ satisfies:
\begin{enumerate}
\renewcommand{\theenumi}{\roman{enumi}}
\item \label{adj.C} If $\Rc$ has type $C_\ell$ for $\ell \ge 3$, then $f \ne f_\ell$.
\item \label{adj.D} If $\Rc$ has type $D_\ell$ for $\ell \ge 4$, then $\deg f \ne \ell$.
\item \label{adj.A3} If $\Rc$ has type $A_3$, then $f = f_3$ (hence $\deg f = 4$).
\end{enumerate}
If $k$ is infinite, then the naive stabilizer of $f$ has identity component $\Gc(k)$.  If $\car k \ne 2, 3$ and does not divide $\deg f$, then the scheme-theoretic stabilizer of $f$ has identity component $\Gc \times k$.
\end{thm}

\begin{proof}
We reduce the proof to Lemma \ref{slem}.  Because the characteristic is very good, $\Lie(\Gct)$ is a restricted irreducible representation of $\Gc$ \cite{Hiss}, and is tensor indecomposable \cite[1.6(i)]{seitzmem}.  This representation appears only in lines $I_1$, $I_2$, $I_4$, and $I_6$ of Seitz's table. 

Line $I_1$ says that $\PSp_{2n}$ is contained in $S^2 \SL_{2n}$; this larger group has a unique invariant of degree $2n$, hence exception \eqref{adj.C}.  Line $I_2$ says that $\SO_{2n+1}$ is contained in $\wedge^2 \SL_{2n+1}$, which has no nonconstant invariants so this gives no exceptions.  Line $I_4$ says $\PSO_{2n}$ is contained in $\wedge^2 \SL_{2n}$ which has a degree $n$ invariant, hence exception \eqref{adj.D}.  Finally, line $I_6$ says that $\PGL_4$ is contained in $\wedge^2 \SL_6$, which has a degree 3 invariant, hence exception \eqref{adj.A3}.

Suppose now that additionally $\car k \ne 2, 3$.  The restriction of $f$ to $\Lie(\Tct) \ot k$ cannot be in the $k$-span of $h^p q^r$ for some nonconstant $h \in k[\Lie(\Tct) \ot k]$ and $r  \ge 1$, for if it were then $h^p$ would also belong to $k[\Lie(\Tct) \ot k]^W$ which would contradict the fact that the restriction of $f$ is a generator.

Applying Lemma \ref{slem} completes the proof of the proposition in case $k$ is algebraically closed.  The claim for arbitrary $k$ follows.
\end{proof}

Although type $C_2$ is excluded from the theorem, in that case $\Gc = \PSp_4 = \SO_5$ is the identity component of the stabilizer of the degree 8 homogeneous polynomial $f_1 f_2$, the degree 4 (but inhomogeneous) $f_1 + f_2$, or anything in $k[f_1, f_2] \setminus (k[f_1] \cup k[f_2])$.

\begin{eg}[$E_8$ octic]
In the case of $E_8$, we choose $p_1, \ldots, p_8$ as in Example \ref{flat}, so that their images over $\R$ are rational multiples of the flat basis $\hat{q}_1, \ldots, \hat{q}_8$ from \cite[p.~15]{Talamini}.  Taking $p_1 := \hat{q}_1$ and $p_2 := 15\hat{q}_2/8$ gives indivisible polynomials with integer coefficients --- to see this it is helpful to refer to \cite[p.~1089]{Mehta}.  In particular,
\[ 
p_2 = 8x_1^8 - 28x_1^4 x_2^4 -14x_1^4 x_2^2 x_3^2 + \cdots \eand p_1^4 = x_1^8 + 19x_1^4 x_2^4+ 72 x_1^4 x_2^2 x_3^2 + \cdots
\]
where the $x_1, \ldots, x_8$ are a basis for the weight lattice defined in \cite{Mehta}.
From this, it is clear that the image of $p_2$ in $k[\Lie(\Tct) \ot k]^W$ is not in the $k$-span of $p_1^4$, not even when $\car k = 2, 3$ which we have excluded.  (Although $p_2 - 8p_1^4$ is divisible by 5.)  The pullback $f$ of $p_2$ to $\Lie(\Gct) \ot k$ then provides an octic form whose scheme-theoretic stabilizer has identity component $E_8 \times k$, and this octic form is \emph{canonical} in the sense that it is determined up to multiplication by a unit in $\Z[1/30]$ by the property of restricting to be an element in the flat basis for the Weyl invariants.
\end{eg}

\section{3875-dimensional representation of $E_8$} \label{E8.3875.sec}

\begin{lem}
Let $G$ be a simple algebraic group of type $E_8$ over a field $k$ and let $V$ be the second smallest  faithful  irreducible representation of $G$.  There exist nonzero $G$-equivariant bilinear maps $b\!: V \times V \to k$ and $\star \!: V \times V \to V$ satisfying
\[
v_1 \star v_2 = v_2 \star v_1 \eand b(v_1 \star v_2, v_3) = b(v_{\pi(1)} \star v_{\pi(2)}, v_{\pi(3)})
\]
for all $v_1, v_2, v_3 \in V$ and every permutation $\pi$ of $\{ 1, 2, 3 \}$, and these properties determine $b$ and $\star$ uniquely up to multiplication by an element of $k^\times$. 
\end{lem}

The representation $V$ in the lemma has dimension 3875 if $\car k \ne 2$ and 3626 if $\car k = 2$, cf.~\cite[A.53]{luebeck}.  In either case the highest weight is the one denoted $\omega_1$ in \cite{Bou:Lie}.

\begin{proof}
Put $\E_8$ for a split semisimple group scheme of type $E_8$ over $\Z$ and $\V$ for a Weyl module of $\E_8$ over $\Z$ with highest weight $\la_1$.  Then $\E_8 \times \C$ is the complex group $E_8$ and $\V \ot \C$ is its second smallest faithful  irreducible representation.
Note that $\V \ot \C$ is orthogonal and has a unique $E_8$-invariant line in $(V\ot\C)^{\otimes 3}$; this line consists of symmetric tensors.  It follows that the same is true for the representation $\V \ot \Q$ of $\E_8 \times \Q$, and we find a symmetric bilinear form $b$ on $\V$ and a bilinear map $\star \!: \V \times \V \to \V$ which are both indivisible and $\E_8$-equivariant.  We can interpret $\star$ as a (not necessarily associative) product operation on $\V$, and we define corresponding operations on $\V \ot k$ by base change.  Because the invariant tensor in $(\V \otimes \C)^{\otimes 3}$ is symmetric, the displayed equations hold in case $k = \C$, and it follows by base change that they hold also for arbitrary $k$.  

Suppose $G$ is $k$-split.  
If $\car k \ne 2$, then $\V \ot k$ is irreducible, i.e., is $V$ (because $G$ is split), and the proof is complete.
If $\car k = 2$, then $\V \ot k$ is reducible and the second displayed identity implies that the maximal proper submodule of $\V \ot k$ is an ideal for the multiplication $\star$.  It follows that $\star$ and $b$ factor through to give a multiplication and a nondegenerate symmetric bilinear form on the irreducible quotient $V$, both of which are nonzero and $(\E_8 \times k)$-invariant.

In the general case, $G$ is isomorphic to a Galois twist of $\E_8 \times k$ by a 1-cocycle $z \in Z^1(k, \E_8 \times k)$.  Using $z$ and Galois descent from a separable closure of $k$ gives $G$-equivariant maps $b$ and $\star$ defined on $V$ over $k$.  Uniqueness and the identities follow because they hold after base change to a separable closure.
\end{proof}

We offer the following observations about the multiplication $\star$ on $V$.  If $\car k \ne 2, 3$, then the automorphism group scheme of $(V, \star)$ is $G$ by Lemma \ref{slem}.

If $\car k = 3$, then Lemma \ref{slem} gives that the automorphism group of this multiplication has $k$-points $G(k)$.  

If $\car k = 2$, then by \cite{seitzmem}, the only other closed connected overgroups of $\E_8(k)$ in $\SL(V)$ are $\Y = \SO(V)$ or $\Sp(V)$, but these cannot stabilize $\star$: the highest weight $\lambda$ of the defining representation of such an $\Y$ is not in the root lattice but $2\lambda$ is, so there is no nonzero $\Y$-invariant multiplication. Therefore $G(k)$ is the naive automorphism group of the multiplication $\star$ on $V$.  Alexander Premet asks: Does this multiplication satisfy the Jacobi identity?

\section{Simple groups as stabilizers of cubic forms} \label{cubic.sec}

Groups of type $B$ and $D$ over an algebraically closed field $k$ are isogenous to $\SO_n$ for some $n$, i.e., they are isogenous to the identity component of the stabilizer in $\GL_n$ of a quadratic form.  
Analogous statements hold for type $E_6$ with a cubic form and type $E_7$ for a quartic form (as long as $\car k \ne 2$).  What about types $C$, $G_2$, $F_4$, $E_8$, and also $A$?  We observe now that all of these, and $E_7$ also, can almost always be obtained as stabilizers of \emph{cubic} forms. 
This result is new even in the case  $k = \C$.  

\begin{thm} \label{for:cubic0}  Let $G$ be a simple and simply connected algebraic group over an algebraically closed field $k$ with (a) $\car k = 0$ or (b) $\car k > 2\rank G + 1$.  There exists an irreducible $kG$-module $V$ and 
a  homogeneous polynomial $f \in k[V]$ of degree   $3$ such that the image of $G$ in $\GL(V)$ is the identity component of the scheme-theoretic stabilizer of $f$.
\end{thm}

At the cost of replacing cubic forms in some cases with quadratic forms, we can ease the hypothesis on the characteristic.

\begin{thm} \label{cubic}
Let $G$ be a simple and simply connected algebraic group over an algebraically closed field $k$ of characteristic
$p \ne 2$.  
Assume that if $G$ is of type $A_{n-1}$ or $C_n$, then $p$ does not divide $n$. 
There exists an irreducible and tensor indecomposable $kG$-module $V$ and
a  homogeneous polynomial $f \in k[V]$ of degree $2$ or $3$ such that the image of $G$ in $\GL(V)$ is the identity component of the naive stabilizer of $f$
\end{thm} 

We postpone the proofs of these theorems until after the following examples, which will be used also in the proof of Theorem \ref{distinctinvariants}.

\begin{table}[hbt]
\begin{center}
\begin{tabular}{ccrr}
Group $G$&Highest weight $\la$&$\dim L(\la)$&$\car k$ \\ \hline
$A_2$&$2\la_1$&6&$\ne 2$\\ \hline
$C_4$&$2\la_2$&308&$\ne 2, 3, 5$ \\
&$2\la_4$&594&$\ne 2, 5, 7$ \\ 
&$\la_2 + \la_4$&792&$\ne 2, 3, 7$ \\ \hline
$E_7$&$2\la_1$&7371&$\ne 2, 5, 19$\\
&$\la_2 + \la_7$&40755&$\ne 2, 3, 7$\\ \hline
$E_8$&$\la_1$&3875&$\ne 2$ \\ \hline
$F_4$&$\la_4$&26&$\ne 3$ \\
&$\la_1 + \la_4$&1053&$\ne 2$ \\ \hline
$G_2$&$2\la_1$&27&$\ne 2, 7$ \\
&$2\la_2$&77&$\ne 2, 3$ \\ \hline
\end{tabular}
\caption{Examples of irreducible representations $L(\la)$ over a field $k$ such that $L(\la)$ has a nonzero $G$-invariant cubic form and there is no overgroup $H$ that stabilizes a cubic form and lies properly between $G$ and $\SL(L(\la))$.} \label{cubic.table}
\end{center}
\end{table}

\begin{eg}[trace zero matrices] \label{Lie.SL}
Consider the conjugation action of $\SL_n$ (equivalently, $\PGL_n$) on the space $M$ of $n$-by-$n$ matrices over an algebraically closed field $k$, for some $n \ge 2$.  Because the matrices with distinct eigenvalues are dense and the normalizer of the diagonal matrices equals the monomial matrices in $\SL_n$, the ring $k[M]^{\SL_n}$ equals the symmetric polynomials in $n$ variables.  That is, it is a polynomial ring with generators of degrees $1, 2, \ldots, n$, the coefficients of the characteristic polynomial.

Put $M_0$ for the Lie algebra of $\SL_n$, i.e., the trace zero subspace of $M$.
Tracking the proof of \cite[4.1]{Nakajima} shows that $k[M_0]^{\SL_n}$ is polynomial with generators the restrictions of the generators of $k[M]^{\SL_n}$ of degrees $2, 3, \ldots, n$.

Finally, suppose  that $\car k \mid n$ and put $\Mb_0$ for $M_0$ modulo the scalar matrices.  Then $k[\Mb_0]^{\SL_n}$ is identified with the subring of elements $f \in k[M_0]^{\SL_n}$ such that $f(t I_n + m) = f(m)$ for all $m \in M_0$ and $t \in k$.  It is easy to see that this ring has transcendence degree $n-2$.  

For later use, we note that in case $k$ has characteristic 2 and $n = 4$, one checks that the coefficients $c_2$, $c_3$ of degrees 2, 3 of the characteristic polynomial belong to $k[\Mb_0]^{\SL_4}$, but that $\det(tI_4 + m) = t^4 + (\text{terms of lower degree in $t$})$.  So $k[\Mb_0]^{\SL_4} = k[c_2, c_3]$.
\end{eg}

\begin{rmk}
The previous example noted that $k[\mathfrak{sl}_n]^{\PGL_n}$ is a polynomial ring. In contrast, $k[\pgl_n]^{\PGL_n}$ is isomorphic to the Weyl-group-invariant subspace of the symmetric algebra on the $A_{n-1}$ root lattice tensored with $k$ \cite[p.~199]{SpSt}.  That is, with the $S_n$-invariant subalgebra of the symmetric algebra on the obvious $(n-1)$-dimensional subspace of the $n$-dimensional permutation representation of $S_n$.  When $\car k$ divides $n$ and $n \ge 5$, this ring is not polynomial by \cite[5.2]{KemperMalle} or \cite[4.3]{Nakajima}, cf.~Problem II.3.18 in \cite{SpSt}.
\end{rmk}

\begin{eg}[self-adjoint endomorphisms]  \label{self.1}
Let $k$ be an algebraically closed field.  Fix a $2n$-dimensional $k$-vector space $W$ for some $n \ge 3$ and a nondegenerate alternating bilinear form $b$ on $W$; write $\Sp(W)$ for the isometry group of $b$.  Define $Y$ to be the space of \emph{self-adjoint} endomorphisms of $W$, i.e., the collection of endomorphisms $T$ so that $b(Tw, w) = 0$ for all $w \in W$.  Note that $\Sp(W)$ acts on $Y$ by conjugation, cf.~\cite[\S2]{GG}.

It is shown in \cite[Th.~2.7]{GG} that any self-adjoint operator leaves invariant a pair of totally singular
complementary spaces with respect to $b$.   With respect to an appropriate basis, 
it follows that a self-adjoint operator
corresponds to $\mathrm{diag}(A,A^{\top})$ and the stabilizer of a pair of such spaces in $\Sp(W)$ is $\GL_n$
which acts via conjugation.  Thus, the $\Sp(W)$-orbits in $Y$
can be identified with similarity classes of $n$-by-$n$ matrices, and $k[Y]^{\Sp(W)}$ is 
generated by the coefficients of the characteristic polynomial of a generic self-adjoint operator,
i.e., it is a polynomial ring generated in degrees $1, \ldots, n$.   A generic element will be one
in which the minimal polynomial of $A$ has degree $n$ and has distinct roots.  (Such an element
has stabilizer isomorphic to $\SL_{2}^{n}$.)
It follows that $k[Y]^{\Sp(W)}$ is a polynomial ring in generators of degrees $1, \ldots, n$.

Write $Y_0$ for the subspace of $Y$ on which the linear invariant vanishes.  The same proof shows that $k[Y_0]^{\Sp(W)}$ is isomorphic to the ring of $\PGL_n$-invariant functions on $M_0$ as in the preceding example.

Now let $X$ denote the space of alternating $2n$-by-$2n$ matrices, i.e., those matrices $L$ so that $L_{ij} = -L_{ji}$ and $L_{ii} = 0$ for all $i, j$.  Then $G$ acts on $X$ via $g \cdot L = gLg^\top$ and this representation is isomorphic to $Y$, via sending $L$ to $JL$ where $J$ is the alternating matrix defining $b$, cf.~\cite[\S2]{GG}.  We write $X_0$ for the submodule of $X$ corresponding to $Y_0$.

We now describe $k[V]^{\Sp(W)}$ for $V$ the irreducible representation with highest weight $\la_2$.
If $n$ is not divisible by $\car k$, then $X_0$ is $V$ and the claim follows from above.
If $n$ is divisible by $\car k$, then $V$ is isomorphic to $X_0/k$ 
 and the pullback $k[V]^{\Sp(W)} \hookrightarrow k[X_0]^{\Sp(W)}$ identifies $k[V]^{\Sp(W)}$ 
with the ring of $\PGL_n$-invariants on the space $\Mb_0$ from the previous example.
\end{eg}

\begin{proof}[Proof of Theorem \ref{cubic}]
By Lemma \ref{slem}, it suffices to produce, for $G$ of each of the types $A_2$, $A_n$ ($n \ge 4$), $C_n$ ($n \ge 3$), $E_7$, $E_8$, $F_4$, and $G_2$, a \emph{restricted} dominant weight $\la$ so that the corresponding irrep $L(\la)$ has a $G$-invariant cubic form and $\la$ does not appear in Seitz's Table 1.

For the groups $G$ and weights $\la$ listed in Table \ref{cubic.table}, the Weyl module of highest weight $\la$ is irreducible (with the restrictions on the characteristic in the table) by \cite{luebeck} and the corresponding irrep for a split complex Lie group has a $G$-invariant cubic form.  Therefore, the irrep $L(\la)$ of $G$ over $k$ also has a $G$-invariant cubic form and we are done with this case.

Example \ref{Lie.SL} gives a restricted irrep for type $A_n$ ($n \ge 4$, $p$ not dividing $n+1$)  with an invariant cubic form, and this irrep does not appear in Seitz's Table 1; this proves Theorem \ref{cubic} for those groups.  Similarly, Example \ref{self.1} handles type $C_n$ for $n = 3$ and $n \ge 5$
(for $p$ not dividing $n$). 
\end{proof} 

The theorem holds also in characteristic 2 for  many types   by the same proof.  However, the argument fails in particular for types $A_1$, $C_4$, and $G_2$.  For example, the only nontrivial restricted irrep of $A_1$ does not support any invariant nonconstant forms, so in this case one must consider irreps that are tensor decomposable.   

\begin{proof}[Proof of Theorem \ref{for:cubic0}]
In Theorem \ref{cubic}, the only places where the polynomial is quadratic is for types $B_n$ ($n \ge 1$) and $D_n$ ($n \ge 3$).  Such a group is isogenous to $\SO_r$ for $r = 3$ or $r \ge 5$.  By the hypothesis on the characteristic, the irrep $V$ with highest weight $2\la_2$ is the vector space of trace zero $r$-by-$r$ symmetric matrices where $\SO_r$ acts by conjugation.  The degree 3 coefficient of the characteristic polynomial is invariant under $\SO_r$ and $V$ does not appear in Seitz's Table 1, so Lemma \ref{slem} gives that $G$ is the identity component of the (naive) stabilizer and further that the scheme-theoretic stabilizer is smooth.
\end{proof}

\section{There are only finitely many overgroups} \label{finitely}

For the proof of the main result Theorem \ref{MT}, we need the statement that 
a semisimple irreducible subgroup of $\SL(V)$ is contained in only finitely many closed subgroups of $\SL(V)$.  We prove instead Proposition \ref{over}, which is much more general.  For $x, y$ in a group $G$, we write $^xy := xyx^{-1}$ and for $A \subseteq G$ we put $^Ay := \{ {^ay} \mid a \in A \}$ and $^xA := \{ {^xa} \mid a \in A \}$. 

\begin{lem}  \label{lem:fixed}  Let $\X < \Y < X$ be groups.
If $[N_X(\X):\X]$ is finite and the number of $\Y$-conjugacy classes
of subgroups $^x\X$, $x \in X$ which are contained in $\Y$ is finite, then
$\X$ has finitely many fixed points on $X/\Y$.
\end{lem}

\begin{proof}  Suppose that $\X x^{-1} \Y=x^{-1} \Y$, equivalently
 $^x\X < \Y$.  Let $\X_i= {^{x_i}\X}, 1 \le i \le m$ be representatives for the
$\Y$-conjugacy classes of $\{{^x\X} \mid {^x\X} < \Y\}$.   So  $^x\X = {^{h_ix_i}\X}$
for some $i$ and some $h_i \in \Y$.   Thus $y:= x_i^{-1} h_i^{-1} x$ is in in $N_X(\X)$.  

Further, $y\X x^{-1}\Y=\X yx^{-1}\Y=\X x_i^{-1}\Y$ and so $N_X(\X)x^{-1}\Y=N_X(\X)x_i^{-1}\Y$.
Of course $N_X(\X)x_i^{-1}\Y$ is a finite union of  $\X\backslash X /\Y$ double cosets,
whence the result.
\end{proof}

\begin{prop}  \label{over}
Let $X$ be a semisimple algebraic group over an algebraically closed field $k$.  If $G$ is a closed, connected
subgroup of $X$ not contained in a proper parabolic subgroup of $X$, then
$G$ is semsimple, $C_X(G)$ is finite, $[N_X(G):G]$ is finite, and there are only finitely many closed subgroups $H$ of $X$ with $G \le H \le X$.
\end{prop}

\begin{proof}
As $G$ is not contained in a proper parabolic subgroup of $X$, it cannot normalize a nontrivial unipotent subgroup, nor centralize a nontrivial torus in $X$ \cite[Th.~4.15a]{BoTi}.  This implies that $G$ is  semisimple and that the identity component of $C_X(G)$ is trivial, hence $C_X(G)$ is finite.
Since the outer automorphism group of a semisimple group is finite, it follows that $[N_X(G):G]$ is finite.

We next show 
that for any connected, closed subgroup $H$ such that $G \le H \le X$, the following holds
(see also \cite[Lemma 11]{martin:new}):
\begin{equation} \label{over.claim}
\text{The set $\{ {^xG} \mid {^xG} \le H \}$ is a finite union of $H$-conjugacy classes.}
\end{equation}

First assume that $k$ is not the algebraic closure of a finite field.
Then we can choose $g_1, g_2 \in \X$ so that $\X$ is the Zariski closure of
$\langle g_1, g_2 \rangle$.
Let $f:\Y^2/\Y \to X^2/X$ be the morphism of varieties induced by the inclusion of $\Y^2$
  in $X^2$, where  $\Y, X$ act on $\Y^2, X^2$ by simultaneous conjugation.
  By \cite[Th.~1.1]{martin}, $f$ is a finite morphism. Let $p_\Y:\Y^2  \rightarrow \Y^2/\Y$ and $p_X:X^2
  \rightarrow X^2/X$   be the canonical projections.  If two elements in $\Y^2$ generate $\Y$-cr subgroups and
  both belong to the same fiber of $p_\Y$, then they are $\Y$-conjugate by \cite[Th.~3.1]{BMR} or see \cite[Th.~6.4]{Richardson:conj}.

If  $x\in X$ and $^x\X \le \Y$, then since $\X$ is not contained in a parabolic subgroup of
$X$, $^x\X$ cannot be contained in a parabolic subgroup of $\Y$.  In particular,
$^x\X$  is an $\Y$-cr subgroup of $\Y$ and   $p_X(^xg_1, {^xg_2}) = p_X(g_1,g_2)$.
Thus, $p_\Y({^xg_1}, {^xg_2})$ lies in the fiber of $f$ above $p_X(g_1,g_2)$.
Since $f$ is a finite map,  this fiber is finite.   So by the previous paragraph,
$({^xg_1},{^xg_2})$ lies in one of finitely many  $\Y$-conjugacy classes.

If $k$ is the algebraic closure of a finite field, let $k'$ be a bigger algebraically closed
field containing $k$.     If $\X_1(k), \X_2(k)$ are closed subgroups of $\Y(k)$, let
$\X_i(k')$ be the Zariski closures of $\X_i(k)$ in $\Y(k')$.
Then $\{g \in \Y(k') \mid {^g\X_1(k')} = \X_2(k') \}$ is a variety defined over $k$.  So
if there are $k'$-points, there are $k$-points.  So if the $\X_i(k')$ are conjugate in $\Y(k')$,
then the $\X_i(k)$ are conjugate in $\Y(k)$, completing the verification of \eqref{over.claim}.

By Lemma \ref{lem:fixed}, it follows that  the closed overgroups of $G$ in $X$
contain only finitely many subgroups in a given $X$-conjugacy class of subgroups.  

\smallskip
We now prove that there are only finitely many closed subgroups $H$ lying between $G$ and $X$.  By the first paragraph of the proof, any such $H$ is  semisimple. By induction on the codimension of $G$ in $X$, we are reduced to proving that $G$ is contained in only finitely many semisimple maximal subgroups of $X$.  Further, \eqref{over.claim} and the fact that $[N_X(G):G]$ is finite reduce us to showing 
there are only finitely many conjugacy classes of maximal
closed subgroups containing $\X$.  If $X$ is simple, then in fact
$X$ only has finitely many conjugacy classes of maximal closed subgroups, by representation theory in the case of classical groups and by \cite{lieseitz} for
exceptional groups.   

Suppose that $X = X_1 \times X_2 \times \cdots \times X_m$
with $m > 1$ and the $X_i$ are simple.  There are only finitely many conjugacy
classes of  maximal subgroups  of the form
$Y_1 \times \cdots \times Y_m$ for $Y_j=X_j$ for all $j \ne i$ and
$Y_i$ is maximal in $X_i$ and so we can ignore these.

The other possible maximal closed subgroups of $X$ are ``diagonal'', i.e., up to re-ordering 
of the factors in $X$ are 
of the form $Y \times X_3 \times \cdots \times  X_m$ where there is a 
bijective morphism $\phi \!: X_1 \to X_2$ and $Y$ is the image of $X_1$ under $\Id_{X_1} \times\, \phi$.
There are countably many conjugacy classes of such
subgroups (essentially corresponding to Frobenius morphisms and outer automorphisms).
It is straightforward to see that the intersection of any infinite collection of nonconjugate  diagonal
subgroups of $X_1 \times X_2$ is a finite group. (Indeed as long as we take more
than $\Out(X_1)$ such subgroups, the intersection is finite.)   
Thus, $G$ cannot be contained in infinitely many
nonconjugate maximal closed subgroups of $X$ for then up to reordering $G$
would be contained in $A \times X_3 \times \ldots \times X_m$ with $A$ finite
and so in $X_3 \cdots \times \cdots \times X_m$, a contradiction.
\end{proof}

\section{Generic stabilizers}

The purpose of this section is to prove the following theorem.

\begin{thm} \label{generic}  
Let $G$ be a closed, simple, and irreducible subgroup of $\SL(V)$ over an algebraically closed field $k$.
  If $\dim V >  \dim G$  (in particular if $V$ is tensor decomposable), then for a generic $v \in V$, the identity
 component of $G_v$ is unipotent.
\end{thm}

In characteristic $0$, irreducible implies tensor indecomposable, so the hypothesis is that $\dim V > \dim G$, and it is already known that  $G_v$ is finite (i.e., the identity component is trivial) and 
with a small number of exceptions, $G_v$ is itself trivial, see \cite{AVE} or \cite[Th.~7.11]{PoV}. This seems likely
to be true in positive characteristic as well and will be the subject of future work.
For our applications, Theorem \ref{generic} more than suffices.   Note that if $V$ is tensor decomposable,
then since the minimal dimension $d$ of a representation of $G$ satisfies $d^2 > \dim G$, $\dim V > \dim G$.  Note also that if
$\dim G > \dim V$, then $ \dim G_v \ge \dim G - \dim V > 0$ and almost always
$G_v$ will contain a torus for generic $v$. (In characteristic $0$, 
only the natural representation of $\SL_2$ has generic stabilizer with an identity component that is nontrivial unipotent.) 

We give the proof at the end of the section. A key part of the proof is the main result of \cite{Kenneally}  that in most cases, if $\dim V > \dim G +2$, then
the identity component of $G_v$ contains only unipotent elements.  

For $G$ an algebraic group acting on a variety $V$, let $V^x$ be the fixed space of $x$ on 
$V$ and $V(x):=\{v \in V \mid {^gx}\cdot v=v \ \text{for some} \ g \in G\}$.

\begin{lem} \label{lem:basic3}  Let 
$G$ be a reductive algebraic group acting on an irreducible variety $V$ over an algebraically closed field $k$.
\begin{enumerate}
\item \label{Vx.back} If $\dim G_v = 0$ for some $v \in V$,  then for all sufficiently large primes $r$, $\dim V(x) < \dim V$ for every $x \in G$ of order $r$.
\item \label{Vx.all} If $\dim V(x) < \dim V$ for all nonidentity $x \in G$, then for $v$ in a nonempty open 
subvariety of $V$, $G_v$ is trivial.
\item \label{Vx.ss} If  $r$ is prime, $r \ne \car k$,
and $\dim V(x) < \dim V$ for all $x$ of order $r$, then for 
$v$ in a nonempty open subvariety, the identity component
of $G_v$ is unipotent.
\end{enumerate} 
\end{lem}

\begin{proof}   
\eqref{Vx.back}: By Lemma \ref{stab.const} there is a nonempty open subvariety $U$ of $V$ and a positive integer $n$ such that $| G_v(k) | = n$ for all $v \in U$.  For every $v \in U$ and prime $r$ not dividing $n$, the conjugacy class of $x$ does not meet $G_v$, so $\dim V(x) < \dim V$, proving \eqref{Vx.back}.

Now let $X \subset G$ be the union of finitely many conjugacy classes and satisfy $\dim V(x) < \dim V$ for all $x \in X$.  Then as $V$ is irreducible, the finite union $\cup_{x \in X} V(x)$ is contained in a proper closed subvariety $Z$ of $V$, and for every $v$ in the nonempty open set $V \setminus Z$, the stabilizer $G_v$ does not meet $X$.

Suppose that $\dim V(x) < \dim V$ for all nonidentity $x \in G$
and take $X$ to be the union of the nonidentity unipotent elements in $G$ and the elements of order $r$, for some prime $r$ not equal to $\car k$.  As $G$ is reductive, $X$ consists of a finite number of conjugacy classes, and the previous paragraph gives that $G_v$ is finite for generic $v$.  Let $n$ be such that $|G_v| = n$ for $v$ in a nonemepty open subvariety of $V$.  Repeating the argument of the previous paragraph with $X$ the set of elements of $G$ whose order divides $n$ completes the proof of \eqref{Vx.all}.

Taking $X$ to be the set of elements of $G$ of order $r$ gives \eqref{Vx.ss}.
\end{proof}

Next we note how to pass from characteristic $0$ to positive characteristic
for semisimple elements.  One can obviously generalize the result but
we just state it in the form we need.
Fix a simple Chevalley group $G$ over $\Z$ and a representation of $G$ on $V := \Z^n$ for some $n$.  Fix algebraically closed fields $K$ and $k$ so that $\car K = 0$ and $\car k = p > 0$.

\begin{lem} \label{lem:basic2}  
Maintain the notation of the previous paragraph.
If $G(K)_v$ is finite for generic $v \in V \ot K$, then the identity component of $G(k)_v$ is unipotent for generic $v \in V \ot k$.
\end{lem}

\begin{proof}  Since $G(K)_v$ is  finite for generic $v$, there exists a prime $r \ne p$
so that $G(K)_v$ contains no elements of order $r$; as in Lemma \ref{lem:basic3}\eqref{Vx.back}, $\dim V(x) < \dim V$ for every $x \in G(K)$ of order $r$.

Let $C$ be a conjugacy class of elements of order $r$ in $G(k)$.
This class is actually defined over the ring of algebraic integers $R$ and consequently 
$C(K)$ and $C(k)$ are both irreducible and have the same dimension. 
Choose $x \in C(K) \cap G(R)$.  Consider the morphism from $G \times V^x \rightarrow V$
defined by $\alpha \!:(g,v) \mapsto gv$; the image of this morphism is $V(x)$.  
Note that this map is actually defined over $R$.   
As the image of $\alpha(K)$ is contained in a proper closed subvariety of
$V \ot K$, then the same is true of $\alpha(k)$ as any hypersurface of $V$ can
be defined by an equation $f=0$ for some polynomial $f$ over $R$ and then
we can reduce modulo $p$.    Thus, $\dim V(\bar{x}) < \dim V$ where $\bar{x}$ is the reduction of $x$
and is an element of $C(k)$, and $\dim V(y) < \dim V$ for any element
$y$ of order $r$ in $G(k)$.  Now apply   Lemma \ref{lem:basic3}\eqref{Vx.ss}.
\end{proof}

We need to deal with a few special cases. 

\begin{lem}  \label{twisted}  Suppose the hypotheses of Theorem \ref{generic}.  
If $V \cong W \otimes W'$ for a representation $W$ of $G$,
where $W'$ is a nontrivial Frobenius twist of $W$ or $W^*$, then
$G_v$ is finite for generic $v \in V$.
\end{lem}

\begin{proof} It suffices to consider the case that $G=\SL(W)$, in which case $G$ has finitely many orbits on $\PP(V)$ 
by \cite[Lemma 2.6]{GLMS} and the result follows since $\dim \PP(V) = \dim G$.
\end{proof}

\begin{lem} \label{lem:basic1}  
Let $G$ be an algebraic group acting on an irreducible variety $V$ over an algebraically closed field $k$.  For $x \in G$, 
if $\dim {^Gx} + \dim V^x  < \dim V$,  then $V(x)$ is contained in a proper closed
subvariety of $V$.
\end{lem}

\begin{proof}  Define $\alpha: G \times V^x \rightarrow V$ by 
$\alpha(g,w)=gw$, so the image of $\alpha$ is precisely $V(x)$.
 The fiber over $gw$ contains $(gc^{-1}, cw)$ for $c \in C_G(x)$, so has dimension at least that of $C_G(x)$, 
whence $\dim V(x) \le  \dim  {^Gx} + \dim V^x$.
\end{proof}

\begin{lem}  \label{A1}   Suppose the hypotheses of Theorem \ref{generic} 
and that $G$ has type $A_1$.
If $\dim V > 5$, then $G_v$ is trivial for generic $v$. If
$\dim V=4$ or $5$, then $G_v$  is finite for generic $v$.
\end{lem}

\begin{proof}   Suppose $\dim V > 5$.  For nonidentity $x \in G$,
$\dim V^x \le (1+\dim V)/2 < \dim V  - 2$.  Since $\dim {^Gx} =2$ for any
noncentral $x \in G$,  Lemmas \ref{lem:basic1} and \ref{lem:basic3}\eqref{Vx.all} give that $G_v$ is trivial for generic $v$.

If $\dim V =4$ and $V$ is not a twist of a restricted module, then apply Lemma \ref{twisted}.
Finally, suppose that $\dim V = 4$ or $5$ and $V$ is restricted (in particular, the characteristic 
is at least $5$).  Any nontrivial unipotent element $x$ has a $1$-dimensional fixed space
and so $\dim V(x) \le 3$. Suppose that $x$ is semisimple but non central.   If $\dim V =4$, 
then $\dim V^x \le 1$ unless $x$ has order $3$, whence $G_v$ has exponent $3$
and so is finite.  If $\dim V =5$, then $\dim V^x  =1$ unless $x$ has order $4$ (and so
is acting as an involution on $V$).  Again, we see that $G_v$ is finite. 
\end{proof}

\begin{eg}[$S^2 \SO_n$] \label{lem:so(n)}  
Let $G = \SO_n(k), n \ge 4$ and $V = L(2\omega_1)$, $p \ne 2$.
We claim that the generic stabilizer is elementary abelian of order $2^n$.

 Let $W$ be the natural module for $G$ and consider $V':=\Sym^2(W)$.
If $p$ does not divide $n$, then $V' \cong k \oplus L(2\omega)$.  Thus, we see
the stabilizer of a generic point is the intersection of $G$ with some conjugate
of $G$ in $\SL(W)$.  Since $\SO(W)$ is the centralizer of an involution in $\SL(W)$
and generically the product of two such involutions is a regular semisimple element, 
it follows that the intersection will generically be the group of involutions in a maximal torus.

If $p$ does divide $n$, then $V'$ is a uniserial module for $G$ with $3$ composition
factors with a trivial socle and head.   Let $V''$ be the radical of this module
and so $V \cong V''/k$.  Clearly, a generic point $v \in V''$ corresponds to a nondegenerate
quadratic form.  Thus, the stabilizer of $v$ in $\SL(W)$ is precisely a conjugate of
$\SO(W)$.  Since $p \ne 2$, $\SO(W)$ is a maximal closed subgroup of $\SL(W)$.
Thus, if $g \in \SL(W)$ and $gv - v \in k$, then $g$ normalizes the stabilizer
of $v$.  Since $\SO(W)$ is self-normalizing in $\SL(W)$, this implies that
$g$ already fixed $v$.   The argument of the previous paragraph 
still applies to give the claim.
\end{eg}  

\begin{proof}[Proof of Theorem \ref{generic}]
Suppose first that $\dim V > \dim G + 2$.  
By \cite[p.~15]{Kenneally}, Lemma \ref{twisted}, and Example \ref{lem:so(n)}, it suffices to consider the following cases:
\begin{enumerate}
\renewcommand{\theenumi}{\alph{enumi}}
\item\label{generic.sl8} $G = \SL_8(k) = \SL(W)$ and $V = \wedge^4 W$. 
\item \label{generic.c44} $G=C_4(k)$ and $V=L(\lambda_4)$, $p \ne 2$.
\item \label{generic.d8} $G = \HSpin_{16}$ and $V$ is a half-spin representation of dimension $128$.  
\item\label{generic.sl9}  $G = \SL_9(k) = \SL(W)$ and $V = \wedge^3 W$.
\item \label{generic.c43} $G = C_4(k)$ and $V = L(\lambda_3)$, $p = 3$.
\end{enumerate}
The first four cases follow from \cite{Vinberg:Weyl} for $\car k \ne 2, 3, 5$, as we now explain using \cite{Levy}.  Set $e = 1$ for cases \eqref{generic.sl8}, \eqref{generic.c44}, and \eqref{generic.d8}, and $e = 2$ for \eqref{generic.sl9}.
Then there is a subspace $\fc$ of $V$ such that $\dim \fc \le (1/e) \rank G$ and there is a
finite and surjective morphism of varieties $\fc \to V/G$, see Lemma 2.7 and p.~432 of \cite{Levy} respectively. Therefore,
$\dim V - \dim G + \dim G_v  \le (1/e) \rank G$.
For each of the possibilities for $G$ and $V$, one checks that
$e(\dim V - \dim G) = \rank G$, so $G_v$ is finite for generic $v$ if $\car k \ne 2, 3, 5$.

In any characteristic, for \eqref{generic.sl8}, \eqref{generic.d8}, and \eqref{generic.sl9}, the representation is minuscule and so equals the Weyl module of the same highest weight over $k$, and Lemma \ref{lem:basic2} gives that generic stabilizers have unipotent identity component.  For \eqref{generic.c44} and $p \ge 5$, $L(\lambda_4)$ is the Weyl module and the same argument works; for $p = 3$, we find that the generic stabilizer for the Weyl module $V'$ has unipotent identity component, and it is easy to see that the same holds for $L(\omega_4) = V'/k$.

For \eqref{generic.c43}, $V = L(\omega_3)$ has dimension 40.
 Let $W$ be the natural $8$ dimensional module for
$G$.  Then $\wedge^3 W$ has composition factors   $W,W, V$.
Now compute for each involution $x \in \Sp(W)$ that
$\dim V^x + \dim {^Gx} < \dim V$.  Thus, for a generic $v$,
$G_v$ contains no involutions and so no torus.

\smallskip
We are reduced to considering the cases $\dim V = \dim G + 1$ or $\dim G + 2$.
First suppose that $G$ is tensor decomposable.  
If $d$ is the minimal dimension of an irreducible representation
of $G$, then $\dim V \ge d^2 > \dim G +2$ unless $G=A_n$
and up to a twist $V \cong W \otimes W'$ where $W$ is the natural
module for $G$ and $W'$ is a Frobenius twist of $W$ or $W^*$.
By Lemma \ref{twisted}, the result holds in this case.

So we may assume that $V$ is tensor indecomposable and in particular
is restricted if $p \ne 0$.  Inspection of the tables in \cite{luebeck}
leave  the following possibilities:

\begin{enumerate}
\renewcommand{\theenumi}{\alph{enumi}}
\setcounter{enumi}{5}
\item\label{generic.a1} $G=A_1$ and $V=L(3\omega_1)$ or $L(4\omega_1)$ with $p \ge 5$.
\item\label{generic.a2} $G =A_2$, $V=L(3\omega_1)$ and $p \ge 5$. 
\item\label{generic.a3} $G=A_3$, $V=L(\omega_1 + \omega_2)$ and $p=3$. 
\item\label{generic.b2}  $G=B_2$, $V=L(\omega_1 + \omega_2)$ and $p=5$.  
\end{enumerate}
In case \eqref{generic.a1}, $G_v$ is finite for generic $v$ by Lemma \ref{A1}.
In case \eqref{generic.a2} the result follows by Lemma \ref{lem:basic2}.   
In case \eqref{generic.a3},  \cite{cohenwales} shows that $G_v$ is finite.  
It remains to
consider \eqref{generic.b2}.   Note that the nonidentity central element $z$ of $G$ acts nontrivially.   Let $x \in G$
be a noncentral  involution (there is precisely one such conjugacy class).
Since $x$ and $xz$ are conjugate, it follows that $\dim V^x=6 = (1/2) \dim V$.
Since $\dim {^Gx} = 4$, we see that for generic $v \in V$,
$G_v$ contains no involutions, whence the result.
\end{proof} 

\section{Same rings of invariants: examples} \label{same.egs}

In this section, $k$ is a field of characteristic $p \ge 0$.  We give examples of simple algebraic groups $\X < \Y < \GL(V)$ so that $k[V]^\X = k[V]^{\Y}$.

\begin{eg}[$\Spin_{11} \subset \HSpin_{12}$] \label{spin12.eg}
Let $V$ be a half-spin module for $\Y = \HSpin_{12}$ and consider the subgroup $\X = \Spin_{11}$.  Suppose first that $\car k \ne 2$.  Igusa calculates in \cite{Igusa} that $k[V]^\X = k[V]^{\Y} = k[f]$ for a homogeneous quartic form $f$ on $V$.  (Alternatively, this can be understood from the point of view of internal Chevalley modules as in \cite{ABS} and \cite[Th.~4.3(3)]{Ruben:np}.  For the determination of the ring of invariants, it suffices to note that there is an $\Y$-invariant quartic form on $V$, which can be constructed from the root system as in \cite{Helenius}.)

Now let $\car k = 2$.  The representation $V$ of $\Y$ is obtained by base change from a representation defined over $\Z$, and the quartic form as in the previous paragraph reduces to the square of a quadratic form $q$, i.e., $k[V]^\Y$ contains $k[q]$, and in fact we have equality because $\dim k[V]^\Y = 1$ by \cite[6.2]{BGL}.  Now, $\dim k[V]^\X = 1$ by \cite[2.11]{GLMS}, so $k[V]^\X = k[V]^\Y = k[q]$.
\end{eg}

\begin{eg}[$\PGL_3 \subset G_2$ in char.~3] \label{A2G2.eg}
Suppose $p = 3$ and put $V$ for the 7-dimensional irreducible representation of $\Y$ split of type $G_2$.  The short root subgroups of $\Y$  generate a subgroup $\X$ isomorphic to $\PGL_3$---see, e.g., \cite[\S7.1]{CGP}---such that the restriction of $V$ to $\X$ is the irreducible part of the adjoint representation \cite{seitzmem}.  (Alternatively, the inclusion $\PGL_3 \le G_2$ can be viewed from the perspective of octonion algebras as in \cite{MatzriVishne}.)

There is a nonzero $G_2$-invariant quadratic form $q$ on $V$, and $G_2$ acts transitively on the non-vanishing set of $q$ in $\PP(V)$; it follows that $k[V]^{\Y} = k[q]$.  As $A_2$ has finitely many orbits on $\PP(V)$ \cite[2.5]{GLMS}, it follows that $k[V]^\X$ must equal $k[q]$ too.
\end{eg}

\begin{eg}[$\SO_{2n}$ in $\Sp_{2n}$ in char.~2] \label{SOSp.eg}  
Let $W$ be a $2n$-dimensional vector space over $k$ algebraically closed of characteristic 2, for some $n \ge 3$.
Write $q$ for a non-degenerate quadratic form on $W$ and $b$ for the alternating form $b(v, w) := q(v+w) + q(v) + q(w)$ for $v, w \in W$.  Put $\SO(W)$ and $\Sp(W)$ for the special orthogonal group of $q$ and the symplectic group of $b$ respectively.
Set $V$ to be the irrep of $\Sp(W)$ with highest weight $\omega_2$ as in Example \ref{self.1}.  We now sketch a proof that $k[V]^{\SO(W)} = k[V]^{\Sp(W)}$. 

Recall the definition of the space $Y$ of self-adjoint operators from Example \ref{self.1}.  We claim that for semisimple self-adjoint operators $T$, $\Sp(W)\cdot T = \SO(W) \cdot T$.  To see this, note that $T$ is a polynomial in a semisimple self-adjoint operator $S$ whose minimal polynomial has degree $n$.  Thus $W = W_1 \perp W_2 \perp \ldots \perp W_n$ where the $W_i$'s are 2-dimensional nondegenerate (and are the distinct eigenspaces of $S$).  Since $\SO(W)$ is transitive on such decompositions,  the
$\SO(W)$ and $\Sp(W)$ orbits of $T$ agree. 
As the semisimple elements are dense in $Y$, it follows that $k[Y]^{\Sp(W)} = k[Y]^{\SO(W)}$, and the same proof shows that the rings of invariants also coincide for $Y_0$ and if $n$ is even for $Y_0/Y_0^{\Sp(W)}$. 

Note that this representation of $\SO(W)$ is the irreducible part of the adjoint module, and that roughly speaking the generators of the ring of invariants have half the degree one finds in characteristic zero.
\end{eg}

\begin{eg}[$F_4$ in char.~2] \label{F4.eg}
Let $\Y = F_4$ and take $V$ to be the $26$-dimensional irreducible representation for $k$ a field of characteristic 2.  Then $\Y$ contains subgroups $\X$ of type $C_4$ and adjoint $D_4$, and the restriction of $V$ to $\X$ is the representation studied in the previous example, see \cite[\S4]{seitzmem}.
Viewing $V$ as the space of trace zero elements in an Albert algebra $A$ and $\Y$ as the automorphism group of $A$, we see that $\Y$ preserves the coefficients of the generic minimal polynomial on $A$ whose restrictions to $V$ are algebraically independent functions of degree 3 (the norm) and 2 (sometimes called the ``quadratic trace'').  The last paragraph of Example \ref{Lie.SL} gives that $k[V]^\X$ is a polynomial ring with generators of degrees 3 and 2; as it is contained in $k[V]^{\Y}$, we conclude that $k[V]^{\Y} = k[V]^\X$.
\end{eg}

\begin{eg}[tensor decomposable in positive characteristic]  \label{twisted.eg} 
Let
$\X=\SL(W)$ where $W$ is a $k$-vector space and $\car k \ne 0$.
 Let $\sigma$ be a Frobenius twist on $\Y$.
Let $V$ be either $W \otimes W^{\sigma}$ or $W^* \otimes W^{\sigma}$.  
Then $\X$ is contained in $\Y = \SL(W) \otimes \SL(W)$ and $\X$ and $\Y$ both act
irreducibly on $V$.  Identifying $V$ with $\mathrm{End}(W)$, we see
that $\Y$ leaves invariant $\det$.  Since $\X$ has a dense orbit on
$\PP(V)$ \cite[Lemma 2.6]{GLMS}, it follows that $k[V]^\X=k[V]^{\Y}=k[\det]$.   Note
also that the generic point of $V$ has a finite stabilizer in $G$, as follows
by dimension and the fact that $G$ has a dense orbit on $\PP(V)$.
\end{eg}

\section{Representations with few invariants} \label{dim.sec}

We now prove the following result, which is most interesting in prime characteristic.  We use what is known in characteristic zero (as in \cite{SK}, \cite{Kac:nil}, or \cite{PoV}) in our proof.

\begin{prop} \label{few}
Let $G \le \SL(V)$ be a simple algebraic group over an algebraically closed field $k$, such that $V$ is an irreducible $G$-module.  Then
up to a Frobenius twist or a twist by a graph automorphism, we have:
\begin{enumerate}
\item $k[V]^G = k$ if and only if $(G, V)$ appears in Table \ref{dim0}.
\item $k[V]^G = k[q]$ for a nonzero quadratic form $q$ if and only if $(G, V)$ appears in Table \ref{dim1}.
\end{enumerate}
\end{prop}

\begin{proof}
The dimension of $k[V]^G$ is at most 1 if and only if $G$ has an open orbit in $\PP(V)$.  Therefore, to identify all pairs $(G, V)$ with $k[V]^G = k$ or $k[q]$, it suffices to examine the list of $(G, V)$ with finitely many orbits in $\PP(V)$ from Tables I and II of \cite{GLMS}.  Some of the entries in Table I are excluded because they have $\dim k[V]^G = 1$ (e.g., because they are defined over $\Z$ and $\dim \C[V]^G = 1$) and are not self-dual.  The spin representation of $B_4$ is defined over $\Z$ and $\C[V]^{B_4}$ is generated by a quadratic form, so the same holds over $k$.  The representations $(A_5, \la_3)$ and $(E_7, \la_7)$ from Table I behave like $(D_6, \la_6)$ as described in Example \ref{spin12.eg}; when $\car k = 2$ they belong to Table \ref{dim1}.  The half-spin representation of $D_5$ has an open orbit in characteristic $\ne 2$ by \cite{Igusa} and in characteristic 2 by \cite[2.9]{Liebeck:affine}, so it belongs to Table \ref{dim0}.  The invariants of $(A_1, \la_1 + p^i \la_1)$, $(A_2, \la_1 + \la_2)$, $(C_3, \la_2)$, $(F_4, \la_4)$, and $(B_5, \la_5)$ were determined in \S\ref{same.egs}.  The ring of invariants of $(A_3, \la_1 + \la_2)$ is generated by an octic form \cite{chen}.  The representation $(C_3, \la_3)$ with $\car k = 2$ is in Table II, so its ring of invariants is 1-dimensional; on the other hand, this representation is identified with the spin representation of $B_3$, so it leaves a quadratic form invariant.
\end{proof}

\begin{table}
\begin{center}
\begin{tabular}{ccccc} 
$G$&$V$&$\dim V$&symplectic?&$\car k$ \\ \hline
$A_n$&$\la_1$&$n+1$&no&all \\
$A_n$ (even $n \ge 4$)&$\la_2$&$\binom{n+1}{2}$&no&all \\
$B_n$&$\la_1$&$2n$&yes&$2$ \\
$C_n$&$\la_1$&$2n$&yes&all \\
$D_5$&half-spin&16&no&all \\
$G_2$&$\la_1$&6&yes&2
\end{tabular}
\caption{Simple $G \le \SL(V)$ with $k[V]^G = k$} \label{dim0}

\begin{tabular}{cccc||cccc}
$G$&$V$&$\dim V$&$\car k$&$G$&$V$&$\dim V$&$\car k$ \\ \hline
$B_n$&$\la_1$&$2n+1$&$\ne 2$&$A_1$&$\la_1 + p^i\la_1$ ($i \ge 1$)&4&$= p \ne 0$ \\
$D_n$&$\la_1$&$2n$&all & $A_2$&$\la_1+\la_2$&7&3 \\
$A_1$&$2\la_1$&3&$\ne 2$ & $A_3$&$\la_2$&6&all \\
$A_5$&$\la_3$&20&2&$B_4$&$\la_4$&16&all \\
$B_3$&$\la_3$&8&all & $B_5$&$\la_5$&32&2 \\
$C_3$&$\la_3$&8&2 &$C_3$&$\la_2$&13& 3\\
$D_6$&half-spin&32&2 &$G_2$&$\la_1$&7&$\ne 2$ \\
$E_7$&$\la_7$&56&2 & $F_4$&$\la_4$&25&3
\end{tabular}
\caption{Simple $G \le \SL(V)$ with $k[V]^G = k[q]$ for a nonzero quadratic form $q$} \label{dim1}
\end{center}
\end{table}

\section{Same transcendence degree} \label{same}

The proof of our main result, Theorem \ref{MT}, relies on showing that there are very few inclusions of groups $G < H \le \SL(V)$ where the the rings of invariant functions $k[V]^G$ and $k[V]^H$ have the same 
(Krull) dimension, equivalently $k(V)^G$ and $k(V)^H$ have the same transcendence degree.   We actually prove a stronger result, namely that when the dimensions are the same, the rings are actually the same.

 \begin{thm}  \label{distinctinvariants}
 Suppose that $G < H \le \SL(V)$ with $G$ a simple algebraic group over an 
algebraically closed field $k$
acting irreducibly on $V$, and $H$ closed in $\SL(V)$.
If $\dim k[V]^G= \dim k[V]^H$, then $k[V]^G= k[V]^H$ and one of the following holds,  up to a Frobenius twist and/or a twist by a graph automorphism:
\begin{enumerate}
\renewcommand{\theenumi}{\alph{enumi}}
\item \label{di.SL} $H = \SL(V)$, $(G,V)$ is in Table \ref{dim0} and $k[V]^G=k$;
\item \label{di.Sp} $H = \Sp(V)$,  $\car k = 2$, $G = G_2$, $\dim V = 6$, and $k[V]^G=k$;
\item \label{di.SO} $H = \SO(V)$,  $(G,V)$ is in Table \ref{dim1} and $k[V]^G=k[q]$; or
\item \label{di.table} $(G,H,V)$ is in Table \ref{sameinv}.
\end{enumerate}
\end{thm}

The case $k = \C$ of the theorem, under the additional hypothesis that $k[V]^G = k[V]^H$, was previously investigated in \cite{Solomon:irredsame} and \cite{Schwarz:lmpi}.   Of course, in the positive characteristic case, we can always replace $V$ by a Frobenius twist (this does change the module
but not the subgroup of $\SL(V)$).  We will ignore this distinction in what follows.  In particular, if $V$ is tensor indecomposable, we will assume
that $V$ is restricted. 

 We remark that in cases \eqref{di.SL}--\eqref{di.table}, we do have $k[V]^G = k[V]^H$. 

In our theorem, we find a fortiori that in almost all of the examples with $\dim k[V]^G = \dim k[V]^H$, the rings of invariants are polynomial rings.  Such representations have been called variously regular, coregular, or cofree; they have been classified in characteristic zero in the papers \cite{KPV}, \cite{Schwarz:regular}, etc., cf.~the books \cite{PoV} or \cite{Popov:GGSO}.  The proofs below do not rely on the full strength of those results, but rather only on the determination of those representations with $\dim k[V]^G \le 1$, which can be found  in \cite{SK} or \cite{Kac:nil} in case $k = \C$ or in \cite{GLMS} for arbitrary $k$, see \S\ref{dim.sec}.
Note that if $k$ does not have characteristic $2$, then all examples have $k[V]^G$ of
dimension at most $1$. 

We start with some lemmas.

\begin{lem}  \label{lem:spherical}  Let $k$ be an algebraically closed field with $\Y$
a connected reductive group over $k$.  Suppose that $\X$ is a proper
reductive subgroup of $\Y$ and $U$ is a connected unipotent subgroup
of $\Y$.  Then  $\X U$ is not dense in $\Y$.
\end{lem}

\begin{proof}
 Suppose that $\X U$ is dense in $\Y$; we may assume that $U$ is a maximal connected unipotent subgroup, i.e., the unipotent radical of a Borel subgroup.
By \cite[Th.~A]{Richardson:Matsushima}, $\Y/\X$ is an affine variety, hence
every orbit of $U$ on $\Y/\X$ is closed (as follows  from the Lie--Kolchin Theorem).
So if $U$ has a dense orbit on $\Y/\X$, it has only one orbit, hence $\Y = \X U$.

Then $(U \cap \X)^\circ$ is a \emph{maximal} connected unipotent subgroup of $\X$.  Indeed,
let $V$ be a maximal connected unipotent subgroup of $\X$; so
$V \le U^g$ for some $g \in \Y$.  However, $\Y=\X U$ implies
that $V=\X \cap U^g$ is conjugate in $\X$ to $\X \cap U$.

Thus $\dim \X U = \dim \X + \dim U - \dim V
= \rank(\X) + \dim V + \dim U < \rank(\Y) + 2 \dim U = \dim \Y$.
\end{proof}

 \begin{lem} \label{lem:unipotent}
 Suppose that $\X < \Y < \SL(V)$ where $G$ and $H$ are connected reductive.  If $k[V]^\X$ and $k[V]^\Y$ have the same transcendence
 degree,  then  for generic $v \in V$,  $\X\Y_v$ is dense in $\Y$ and the identity
 component of $\Y_v$ is not unipotent.
 \end{lem}

 \begin{proof}  
 The transcendence degree of $k[V]^\Y$ is the codimension of the highest
 dimensional orbit of $\Y$ on $V$.  Thus, for $v$ generic
 $\X v$ is dense in $\Y v$, whence $\X\Y_v$ is dense in $\Y$.  Now by
 Lemma \ref{lem:spherical}, this cannot happen if the identity component of
 $\Y_v$ is unipotent.
\end{proof}
 
\begin{lem} \label{decomp} 
Suppose that $\X < \Y < \SL(V)$ and $k[V]^\X$ and $k[V]^\Y$ have the same transcendence
 degree.   Assume further that $\X$ is a simple algebraic group and that $V$ is a tensor 
 decomposable  irreducible  $\X$-module.  Then $\X=\SL(W)$, up to a twist
 $V = W \otimes W' $ where $W'$ is a (nontrivial) Frobenius twist of $W$ or $W^*$,
and  $\Y= \SL(W) \otimes \SL(W)$.
 \end{lem}

 \begin{proof}   
  Note that $\Y$ is semisimple and $\dim V$ is not prime.   Suppose that $\Y$ is simple
  and   $V$ is tensor decomposable for $\Y$, then 
  it follows by Theorem  \ref{generic} that the generic stabilizer in $\Y$ of an element of
  $V$ has unipotent identity component, contradicting Lemma \ref{lem:unipotent}.
  
  If $\Y$ is simple and $V$ is tensor indecomposable for $\Y$, then 
  by \cite[Th.~1]{seitzmem},  $\Y$ is $\SL(V)$, $\SO(V)$, or $\Sp(V)$. 
  in particular, $k[V]^{\Y} =k$ or $k[q]$ with $q$ quadratic, so
  $\dim k[V]^{\X} \le 1$ as well.  On the other hand, as we have already
  noted, $\dim V > \dim \X +1$ unless $\X=\SL(W)$
  and $V = W \otimes W'$ where $W'$ is a twist of $W$ or $W^*$.
  In that case $k[V]^{\X} =k[f]$ with $f$ of degree $\dim W$.
  Thus,  $\dim V =4$ and $\Y=\SO(V)$ is tensor decomposable. 
  
  So we may assume that $\Y$ is not simple. 
    
    First consider the case that $\X$ is maximal in $\Y$. 
  It  follows that $\Y$ is a central product of two copies of $\X$
  and $\X$ embeds diagonally in $\Y$.   For convenience we consider
  $\Yt = \X \times \X$.  Let $\pi_i$ denote the projection onto the $i$-th  factor.

   Thus,  $V= W_1 \otimes W_2$
  where the first copy of $\X$ acts trivially on $W_2$ and the second copy
  acts trivially on $W_1$.   Assume that $\dim W_1 \le \dim W_2$.
    A straightforward computation shows that
  for a generic $v \in V$, $\pi_2(\Yt_v)$ is an injection into $\SL(W_2)$.
  Moreover,
  if $\dim W_1 < \dim W_2$,  the $\pi_2(\Yt_v)$ is proper in $\X$.  In particular,
  $\dim \Y_v < \dim \X$, whence $\dim (\X\Y_v) < 2 \dim \X = \dim \Y$.  Thus
  $\dim \X v < \dim \Y v$ and so $\dim k[V]^\X > \dim k[V]^\Y$.

  If $\dim W_1= \dim W_2$, we can view 
  $\X \le \SL(W)$ and $\X \le \SL(W) \le \Y \le Y:=\SL(W) \times \SL(W) < \SL(V)$.
  Note that if $v$ is a generic point of $V$, then $Y_v$ is a diagonal subgroup
  of $Y$.  If $\X$ is proper in $\SL(W)$, then $\dim \Y_v$ is generically less than
  $\dim \X$ since the intersection of $\Y$ with a generic $Y_v$ will not be a full
  diagonal subgroup of $\Y$.   In that case $\dim \X + \dim \Y_v < \dim \Y$,
  whence $\dim k[V]^\X > \dim k[V]^\Y$.
  
  If $\X = \SL(W)$, then (since $V$ is irreducible) we see that $V$ is a twist
  of $W \otimes W'$ where $W'$ is a Frobenius twist of $W$ or $W^*$
  as allowed in the conclusion.

  If $\X$ is not maximal in $\Y$, then we can choose $Y$ with
  $\X < Y < \Y$ with $\X$ maximal in $Y$, whence by induction
 we are in the one case allowed.   Then $Y$ is maximal in 
 $\SL(V)$ \cite[Th.~3]{seitzmem}, a contradiction. 
  \end{proof}
  
  We need one more preliminary result.
  
  \begin{lem} \label{lem:spso}  Let $G=\SL(W)$ with $\dim W = 2m > 2$.
  Then $\SO(W)\,g\Sp(W)$ is not dense in $G$ for any $g \in G$.
  \end{lem}
  
  \begin{proof}  If $G$ is in characteristic $2$, then $\SO(W) \le \Sp(W)$.
  Hence by  \cite{GG},  $\dim (\Sp(W) \cap \,{^g\Sp(W)}) \ge 3m$, whence the result.
  In any other characteristic, we can take $\SO(W)=C_G(\tau)$ where
  $\tau$ is the inverse transpose map and ${^g}\Sp(W) = C_G(\tau J)$
  where $J$ is a skew-symmetric matrix.   Thus,
  $\SO(W)\, \cap\, {^g}\Sp(W) = C_G(\tau, \tau J)$.  Generically, $J$ will be
  a semisimple regular element and so this intersection is the centralizer of
  $\tau$ in $\SL(W) \cap k[J]^\times$, which has dimension $m-1$, whence the result.
  \end{proof}

 \begin{proof}[Proof of Theorem \ref{distinctinvariants}]
 If $V$ is tensor decomposable for $\X$, the result follows by Lemma \ref{decomp}; this case is in Table \ref{sameinv}.  So we
 assume that $V$ is tensor indecomposable
 for $\X$ (and so also for $\Y$). This forces $\Y$ to be a simple algebraic group as well.  
 
If $H = \Sp(V)$, then $(G, V)$ is symplectic and appears in Table \ref{dim0}, giving the claim in \eqref{di.Sp}.  As the representations in \eqref{di.SL} and \eqref{di.SO} have already been listed in Tables \ref{dim0} and \ref{dim1}, we assume that $H \ne \SL(V)$, $\Sp(V)$, $\SO(V)$, and Theorem 1 in \cite{seitzmem} asserts that $(G, H)$ appears in Table 1 in ibid.
 
 If $\dim V > \dim \Y$,
 then $\Y_v^\circ$ is unipotent for generic $v$, a contradiction by Lemma \ref{lem:unipotent}.
 So we may assume that $\dim V \le \dim \Y$, and therefore $V$ is one of the representations enumerated in \cite{luebeck}.

 Suppose that $V$ is the irreducible part of the adjoint
 module for $\Y$.  If $\dim V \ge \dim \Y - 2$, then $\X$ will not act irreducibly.  Otherwise, as in \cite{Hiss}, $\Y$ has type $B$, $C$, or $F_4$ and $p = 2$, or $\Y$ has type $G_2$ and $p = 3$.  Consulting Seitz's table, the only possibility for $\X$ is $\X = D_n$ in case $\Y = C_n$ and $p = 2$; but in this case the representation is a twist of the natural module and we have $k = k[V]^\Y \ne k[V]^\X$.

 We can exclude many cases by exploiting equation \eqref{dim.eq} and its analogue for $H$.  
  Combining these, we obtain:
 \begin{equation} \label{dim.bound}
 \dim \X  + \dim k[V]^\Y \ge \dim V.
 \end{equation}

\case{Exceptional groups} If $\Y$ has type $E_n$ for some $n$, then the only representations in \cite{luebeck} not already considered are the cases $\Y = E_6$ or $E_7$ and $V$ the minuscule representation of $\Y$ of dimension 27 or 56 respectively.  In either case, $\dim k[V]^\Y = 1$.  According to Seitz's table, the only possibility for $\X$ is $C_4$ and $p \ne 2$ (and $\Y = E_6$), in which case the restriction of $V$ to $C_4$ is $L(\omega_2)$ which has $\dim k[V]^{C_4} \ge 2$ by \cite{GLMS}.

 If $\Y \cong F_4$, then the only remaining choice for $V$ from \cite{luebeck} is that $\dim V = 26$ (resp., $25$ if $p=3$).
 Note that the identity component of $\Y_v$ for generic $v$ is $D_4$, so $\dim k[V]^\Y$ is 2 (resp., $1$ if $p = 3$).  By Seitz's table,
 the only possibilities for $\X$ are $G_2$ with $p=7$ (which is too small by \eqref{dim.bound}), or $D_4$ or $C_4$ with $p = 2$ (which are in Table \ref{sameinv}).

  If $\Y \cong G_2$, then $\dim V = 7$ (6 if $p = 2$) is the only remaining possibility, and $\dim k[V]^\Y \le 1$, hence $\dim \X \ge 5$.
  By \cite{seitzmem},
 the only possibility for $\X$ is $A_2$, with $p=3$ and $V$ the irreducible part of the adjoint representation of $\X$, as in Table \ref{sameinv}.

 \smallskip
  Thus, $\Y$ is a classical group of rank at least 2, and $V$ is not the natural module.

\case{$\dim k[V]^\Y \le 1$}
Comparing the tables in \cite{luebeck} and \cite{GLMS} shows that most of the possibilities for $V$ have $\dim k[V]^\Y \le 1$. (This includes such cases as $\Y = \SL_4(k)$ and $V = L(\la_1 + \la_2)$ for $p = 3$.)  Of the possibilities for $\X$ listed in Seitz's table, most fail condition \eqref{dim.bound}.
The only interesting cases are where $\Y$ has type $D$ and $V$ is a half-spin representation, or $\Y = \SL(W)$ and $V = \wedge^2 W$ or $L(2\la_1)$.

So suppose $\Y = \HSpin_n$ and $\X = \Spin_{n-1}$ for $n = 10$ or 14.  If $\car k \ne 2$, Igusa showed in \cite{Igusa} that $\dim k[V]^\X = \dim k[V]^\Y + 1$.
Without restriction on the characteristic, for $n = 14$, $\dim k[V]^\Y = 1$ by \cite[Prop.~6.2]{BGL}, whereas $\dim k[V]^\X \ge 2$ by \cite{GLMS}.  For $n = 10$ and $\car k = 2$, $k[V]^\Y = k$ by \cite[2.9]{Liebeck:affine}, whereas $k[V]^\X$ contains a quadratic form (by reduction from $\Z$).  So these cases do not occur.  The case $n = 12$ is in Table \ref{sameinv}.

Suppose $\Y=\SL(W)$ with $\dim W > 3$ and $V=L(\lambda_2)$. If $\dim W$
is odd, then $k[V]^{\Y} = k$ and there are no possibilities for $\X$ by
Table \ref{dim0}.   Suppose that $\dim W$ is even.   Then $\dim k[V]^H = 1$ and the  only semisimple maximal  subgroups $\X$
 with large enough dimension to possibly satisfy  \eqref{dim.bound} 
 are $\Sp(W)$ or $\SO(W)$.   Of course $\Sp(W)$ is not irreducible on $V$.
   It follows by Lemma \ref{lem:spso}  that $\X\SO(W)$ is not dense
 in $\SL(W)$. 

  Suppose $\Y = \SL(W)$ and $V = L(2 \lambda_1)$.
  As $V$ is restricted, we have $p \ne 2$.
 The generic stabilizer is $\SO(W)$.  There is no semisimple maximal subgroup
 $\X$  of $\Y$ such that $\X\SO(W)$ is dense in $\Y$, whence the result holds in this
 case (the only semisimple subgroup of sufficiently large dimension is $\Sp(W)$
 but  by Lemma \ref{lem:spso}   $\Sp(W)g\SO(W)$ is never dense in $\SL(W)$).

\case{Remaining cases}
We now mop up the remaining representations from \cite{luebeck}.
 Suppose $\Y =\Sp_{2m}(W)$ and $V=L(\lambda_2)$
 --- this is the irreducible part of $\wedge^2W$.  
 By Example \ref{SOSp.eg},  $\dim k[V]^{\Y} = m-1$ or $m-2$, depending upon whether
 $p$ divides $m$ or not.
 Thus, $\dim \X  \ge \dim V - (m-1)$ or $(m-2)$.  The only semisimple maximal
 subgroup of $\Sp(W)$ with such a dimension is $\SO(W)$ with $p=2$. 
 This case is in Table \ref{sameinv}.

 \smallskip
The tables in \cite{luebeck} leave only 
$V = L(\lambda_m)$ for $\Y = \Sp_{2m}(k)$ for $m = 3$ (all $p$), or 4, 5, or 6 ($p = 2$ only).
But for these cases, \cite{seitzmem} shows there are no irreducible simple algebraic subgroups.
\end{proof}

Here is another result  where we allow $\X$ to be semisimple.

\begin{prop}  \label{ss}
Suppose that $\X < \Y  \le \SL(V)$ with $\Y$ a simple algebraic group
and $\X$ a semisimple irreducible subgroup of $\SL(V)$.     If 
$\Y$ acts tensor indecomposably on $V$  and $\dim k[V]^\X = \dim k[V]^\Y$, 
then either $\X$ is simple, or $k[V]^\X=k$, or $\dim V =8$ with
$\X = \Sp_2 \times \Sp_4$ and $\Y = \SO_8(V)$. 
\end{prop} 

\begin{proof}   Assume that $\X$ is not simple. 
By \cite[Th.~1]{seitzmem}, it follows that $\Y=\SO(V)$,  $\Sp(V)$ or $\SL(V)$. 
In the last two cases,  $k[V]^\Y=k$, whence the result holds. 

Thus, we may assume that $\Y=\SO(V)$ and  $\dim V \ge 8$. 
   Then $\X \le Y:=\Sp(W_1) \otimes \Sp(W_2)$ or  $\SO(W_1) \otimes \SO(W_2)$
   and moreover $\X$ has a dense orbit on $\PP(V)$.   Assume that $m_1= \dim W_1 \le  m_2 =\dim W_2$.
   It follows that the stabilizer $Y_v$ of a generic $v \in V$
   has dimension equal to $\dim \SO_{m_2 - m_1}$ in the second case
   and equal to $(3/2)(\dim W_1) + \dim \Sp_{m_2 - m_1}$ in the first case.  
   
   Note that in the second case $m_1 \ge 3$.  Thus, we see that
   $\dim k[V]^Y =\dim V - \dim Y + \dim Y_v > 1 = \dim k[V]^\Y$
   unless we are in the first case with $m_1 = 2$ and $m_2 =2 $ or $4$.
   If $m_2 = 2$, then $\Y = \SO(V)$ is not simple.  
   If $m_2=4$, then in fact $\Sp_2 \otimes \Sp_4$ does have
   a dense orbit on $\PP(V)$ because $V$ is an internal Chevalley module \cite[Table I]{GLMS}.    The only possible proper irreducible subgroup
   of $Y$ is $\SL_2 \times \SL_2$ and is too small to have a dense orbit
   on $\PP(V)$.  
   \end{proof}

\section{Same rings of invariants, but without containment of groups}

We close our discussion of groups with the same invariants with an easy consequence of Theorem \ref{distinctinvariants}, where we drop the hypothesis that $G$ is contained in $H$, but we strengthen the hypothesis on the invariants to be that the rings $k[V]^G$ and $k[V]^H$ are the same. 

\begin{cor}
Assume that $G$ is a simple algebraic group acting
irreducibly on $V$ and that $\dim k[V]^G > 1$.  If
$H$ is a connected algebraic subgroup of $\SL(V)$ acting irreducibly on 
$V$ with $k[V]^H = k[V]^G$, then $H = G$ or one of the following holds,  up to a Frobenius twist and/or a graph automorphism: 
\begin{enumerate}
\item $\car k = 2$,  $H = \Sp_{2n}$, $G = \SO_{2n}$ with $n \ge 3$, and $V = L(\lambda_2)$.
\item  $\car k = 2$, $\dim V = 26$,  $G$ has type $D_4$ or $C_4$, and $H$ has type $F_4$.
\end{enumerate}
\end{cor}

\begin{proof} Let $Y$ be the group generated by $G$ and $H$ (which we assume is distinct from $G$).  Then $G < Y$ have the same
invariants and so Theorem \ref{distinctinvariants} implies that $\car k = 2$ and $Y = \Sp$
and $V = L(\lambda_2)$, or $Y = F_4$ and $\dim V = 26$. Also by Theorem \ref{distinctinvariants}, there is no
other simple $H$ and by Proposition \ref{ss} (since $G$ is simple),
there is no semisimple example either.
\end{proof}

\section{Main theorem} 

We now fix a pair $(G, V)$ and ask whether there is some $f \in k[V]^G$ such that $G$ is the identity component of the stabilizer of $f$.  Trivially, we must exclude those representations where $k[V]^G = k$ or $k[q]$ for a quadratic form $q$ (when $G \ne \SO(V)$); these make a short list that we provide in Tables \ref{dim0} and \ref{dim1}.  These lists are well known in characteristic zero---a convenient reference is the table at the end of \cite{PoV}.

In the following theorem, we write $\kalg$ for an algebraic closure of a field $k$.

\begin{thm} \label{MT}
Let $G < \GL(V)$ be an absolutely simple algebraic group over a field $k$ such that $V$ is absolutely irreducible and $(G(\kalg), V \ot \kalg)$ is not in Table \ref{dim0} nor Table \ref{dim1} up to Frobenius twists and graph automorphisms.  Then exactly one of the two following statements holds:
\begin{enumerate}
\item \label{MT.f} There exists a homogeneous $f \in k[V]^G$ such that the naive stabilizer of $f$ in $\GL(V \ot \kalg)$ 
has identity component $G(\kalg)$.
\item $(G(\kalg), V \ot \kalg)$ is a Frobenius twist of a representation from Table \ref{sameinv}.
\end{enumerate}
\end{thm}

\begin{table}[bt]
\begin{center}
\begin{tabular}{cccccc}
$G$&$H$&$\dim V$&$\car k$&degrees&see \\ \hline
$\Spin_{11}$&$\HSpin_{12}$&32&all&$\begin{cases} \scriptstyle{4}&\scriptstyle{\text{if $\car k \ne 2$}}\\ \scriptstyle{2} & \scriptstyle{\text{if $\car k = 2$}}\end{cases}$& \ref{spin12.eg} \\
$\PGL_3$&$G_2$&7&3&2&\ref{A2G2.eg} \\
$\SO_{2n}$ ($n \ge 3$)&$\Sp_{2n}$&$\begin{cases} \scriptstyle{2n^2 - n - 2}& \scriptstyle{\text{if $n$ even}}\\ \scriptstyle{2n^2 -n-1} & \scriptstyle{\text{if $n$ odd}} \end{cases}$&2&$\begin{cases}?\\2, 3, \ldots, n\end{cases}$& \ref{self.1}, \ref{SOSp.eg} \\
$\SO_8$ or $\Sp_8$&$F_4$&26&2&2, 3& \ref{F4.eg}\\
$\SL_n$ & $\SL_n \otimes \SL_n$ & $n^2$ & $ \ne 0$ & $n$& \ref{twisted.eg}
\end{tabular}
\medskip
\caption{Representations excluded from Theorem \ref{MT}} \label{sameinv}
\end{center}
\end{table}

The theorem gives tight control over the stabilizer of $f$, because the normalizer $N$ of $G$ in $\GL(V)$ is known precisely.  It has identity component $N^\circ$ generated by $G$ and the scalar matrices, and $N$ itself can hardly be much larger.  Indeed, 
there is an inclusion of $N/N^\circ$ into the automorphism group of the Dynkin diagram of $G$, which is $1$ (type $A_1$, $B$, $C$, $E_7$, $E_8$, $F_4$, or $G_2$), $\Z/2$ (type $A_n$ for $n \ge 2$, $D_n$ for $n \ge 5$, or $E_6$), or has order  6 (type $D_4$), see e.g.\ \cite[\S16.3]{Sp:LAG}.  In any concrete case, one is reduced to checking which representatives of $N/N^\circ$ in $N$ stabilize $f$.

The shortness of the proof hides the fact that it relies on all of the results from sections \ref{finitely} through \ref{same}.

\begin{proof}[Proof of Theorem \ref{MT}]
We assume that $(G(\kalg), V\ot \kalg)$ does not appear in Tables \ref{dim0}, \ref{dim1}, nor \ref{sameinv}, and we will produce an $f$ as in \eqref{MT.f}.

Suppose first that $k$ is perfect.
By Proposition \ref{over} there are only finitely many closed connected proper overgroups of $G \times \kalg$ in $\SL(V \ot \kalg)$; write $\cO$ for the set of such.  The $k$-automorphisms of $\kalg$ permute the elements  of $\cO$ so there is a finite Galois extension $L$ of $k$ such that all $L$-automorphisms of $\kalg$ fix every element of $\cO$.  Write $\Y_1, \ldots, \Y_r$ for representatives of the orbits in $\cO$ under the group $\G$ of $k$-automorphisms of $L$. 

As $(G, V)$, by hypothesis, does not appear in Table \ref{sameinv}, Theorem \ref{distinctinvariants} gives that $\dim L[V]^{\Y_i} < \dim L[V]^G$ for each $i$, and we can pick a homogeneous polynomial $f_i \in k[V]^G \setminus k[V]^{\Y_i}$.  Define
\[
f := \prod_{i = 1}^r \prod_{\s \in \G} \s(f_i).
\]
As $\Y_i$ is semisimple, it has no nontrivial characters, so the irreducible factors (in $L[V]$) of a $\Y_i$-invariant function are also $\Y_i$-invariant, hence $f$ is not invariant under any of the $\Y_i$'s.  As $f$ is fixed by every element of $\G$, it belongs to $k[V]$, and the theorem is proved in this case.

Now suppose that $k$ is imperfect.  The case where $k$ is prefect provides an $f \in k^{p^{-\infty}}[V]^G$ as in \eqref{MT.f}.  Writing $f$ as a polynomial in a dual basis for $V$, we find finitely many coefficients in $k^{p^{-\infty}}$ so there is a positive integer $s$ such that $f^{p^s}$ belongs to $k[V]$ and has stabilizer with identity component $G(\kalg)$.
\end{proof}

Theorem \ref{MT} can be compared with the following famous theorem of Chevalley \cite[5.5.3]{Sp:LAG}: if $G$ is a subgroup of a linear algebraic group $G'$, then there is a finite-dimensional representation $V$ of $G'$ so that $G$ is precisely the stabilizer in $G'$ of a point in $\PP(V)$.  Chevalley's result holds in much greater generality ($G$ need not be simple, nor even reductive), but Theorem \ref{MT} says that for a typical irreducible and tensor indecomposable representation $V$ there is a $G$-fixed point $[f]$ in $\PP(\Sym^d(V))$ for some $d$ such that the identity component of the stabilizer of $[f]$ is $G$.

\section{Realizations of simple groups as automorphism groups} \label{arbfield.sec}

There are several general-purpose mechanisms for realizing a semisimple group (up to isogeny, taking identity components, and avoiding some bad characteristics or special cases) as automorphism groups of some algebraic structure, namely as the automorphism group of:
\begin{itemize}
\item the spherical building associated with $G$ \cite{Ti:BN};
\item a twisted flag variety of $G$ \cite{Dem:aut}; or
\item the Lie algebra of $G$ \cite{St:aut}.
\end{itemize}
These three interpretations are easiest to understand in case the group $G$ is split, but it is easy to see via twisting that these interpretations extend to describe also non-split semisimple algebraic groups over any field.  These bullets only give \emph{adjoint} groups.  One can also interpret $G$ as the automorphism group of
\begin{itemize}
\item some finite-dimensional $k$-algebra \cite{GordeevPopov}.  This construction is more precise than the previous one in that it gives $G$ on the nose (and not just up to isogeny) but also less precise in that one does not have control of the algebra.  For $E_8$, this interpretation just gives that $E_8$ is the automorphism group of its Lie algebra, as in the previous bullet.
\end{itemize}
Our Theorem \ref{MT} adds an additional item to this list:
\begin{itemize}
\item a homogeneous function $f$ on a representation $V$ of $G$.  Here one can pick a faithful and absolutely irreducible representation $V$ (if one exists), and get $G$ on the nose (and not just up to isogeny).
\end{itemize}
Although $f$ is not uniquely determined, there is still enough of a connection between $G$ and $f$ so as to play properties of one off against the other.  Here are three examples of such.

First, we can relate isotropy of $G$ (i.e., whether $G$ contains a nonzero split $k$-torus) with isotropy of $f$ (whether there is a nonzero $v \in V$ such that $f(v) = 0$).

\begin{lem}
Suppose $G$ is a linear algebraic group acting with finite kernel on a vector space $V$ over an infinite field $k$ and $f \in k[V]^G$ is homogeneous and non-constant.  If $G$ is isotropic, then $f$ is isotropic.
\end{lem}

\begin{proof}
Put $T$ for the nontrivial $k$-split torus in $G$.  Its image in $\GL(V)$ is also 
a nontrivial split $k$-torus in $\GL(V)$, so there is a nonzero $v \in V$ that is not fixed by the action of $T$.  But $f$ is homogeneous and $T$-invariant, so $f(v) = 0$.
\end{proof}

This relationship between the isotropy of $G$ and isotropy of $f$ is similar to what one finds for the Tits algebras of $G$.  (Recall that the Tits algebras defined in \cite{Ti:R} or \cite{KMRT} are classes in the Brauer group of finite separable extensions of $k$ corresponding to dominant weights of $G$.  In case $G$ is the spin or special orthogonal group of a quadratic form, the only possibly nontrivial Tits algebra is the Brauer class of the Clifford algebra or the even part of the Clifford algebra.)  One knows that if $G$ is isotropic, then the Tits algebras cannot have too large an index, where the precise bounds depend on the maximal split torus in $G$ and the dominant weight corresponding to the Tits algebra.  There is no corresponding converse implication, in that the Tits algebras may all be zero and yet the group can be isotropic, which occurs for example when $G$ is any of the compact real groups $\Spin(8n)$ for $n \ge 1$, $G_2$, $F_4$, or $E_8$.  Nonetheless, the one-way implication between isotropy of $G$ and small indexes for the Tits algebras, notably exploited by Merkurjev in \cite{M:simple} to disprove Kaplansky's conjecture, is now a standard tool in the study of semisimple algebraic groups, quadratic forms, division algebras, etc., over arbitrary fields.  See the survey \cite{Hoff:usurv} or papers such as \cite{Hoff:EL}, \cite{Hoff:pyth}, \cite{Izh:9}, \cite{GS:quats}, and \cite{Meyer:inf}.

Second, Theorem \ref{MT} gives a way to study $k$-forms of $G$.  Recall that an algebraic group $G'$ over $k$ (resp., $f' \in k[V]$)  is called a \emph{$k$-form of $G$} (resp., \emph{of $f$}) if there is an extension field $L/k$ so that $G' \times L$ is isomorphic to $G \times L$ (resp., there is a $g \in \GL(V \ot L)$ so that $f' = f \circ g$ in $L[V]$).
If $f$ is \emph{$k$-similar} to a homogenous form $f'$---i.e., if there is a $g \in \GL(V)$ and $\mu \in \kx$ so that $f' = \mu f \circ g$ in $k[V]$---then obviously $f$ and $f'$ have isomorphic stabilizers in $\GL(V \ot L)$ for every extension $L$ of $k$.  That is, taking identity components of stabilizers gives an arrow
\begin{equation} \label{similarity}
\fbox{\parbox{1in}{$k$-forms of $f$\newline up to $k$-similarity}} \to \fbox{\parbox{1.2in}{$k$-forms of $G$\\ up to $k$-isomorphism}}
\end{equation}
that is functorial in $k$.  This arrow is well known to be injective (but typically not surjective) in the case where $G$ is the special orthogonal group of a quadratic form $f$.  It is bijective in case $G$ has type $E_8$ and $f$ is the octic form studied in sections \ref{E8.sec} and \ref{E8.adj.sec} or the cubic form studied in section \ref{E8.3875.sec}.  In general, the arrow \eqref{similarity} may be injective, surjective, both, or neither, and a careful choice of $V$ and $f$ can guarantee that it is bijective; this question is studied in \cite{BermudezRuozzi}.

Third, for $S$ the scheme-theoretic stabilizer of $f$, faithfully flat descent trivially gives a bijection
\begin{equation} \label{torsors}
H^1(k, S) \leftrightarrow \fbox{$k$-forms of $f$ up to $k$-isomorphism}
\end{equation}
that is functorial in $k$, where $H^1$ denotes flat cohomology.  By examining the number of independent parameters appearing in explicit $k$-forms of $f$, one can in principle give an upper bound on the essential dimension of the group $S$ as defined in \cite{Rei:ICM} or \cite{Rei:whatis}.  This is a usual method for giving an upper bound on the essential dimension of an orthogonal group.  Our results here give an effective means for describing $S$ as an algebraic group, thereby allowing one to use \eqref{torsors} to prove statements about essential dimensions of familiar groups.  For example, can studying $k$-forms of an octic as in section \ref{E8.sec} or \ref{E8.adj.sec} give a better upper bound on the essential dimension of $E_8$ over $\C$?
The strongest result currently known is that the essential dimension is at most 231 \cite[Cor.~1.4]{Lemire}, which is quite far from the lower bound of 9 \cite{Rei:ICM}.

We expect that the homogenous forms provided by our Theorem \ref{MT} will provide many new avenues for studying simple algebraic groups over arbitrary fields, since the three relationships we have chosen to highlight here are analogous to previously known tools that have already been widely exploited.

\medskip
{\small\textbf{Acknowledgements.} {We thank Brian Conrad, Benedict Gross, Ben Martin, George McNinch, Gopal Prasad, Alexander Premet, and Jean-Pierre Serre  for their comments.  Garibaldi's research was partially supported by NSF grant DMS-1201542, the Charles T.~Winship Fund, and Emory University.
Guralnick was partially supported by the NSF
  grants DMS-1001962, DMS-1302886 and  Simons Foundation fellowship 224965.
  He also thanks IAS, Princeton. 
  }}

\newcommand{\etalchar}[1]{$^{#1}$}
\providecommand{\bysame}{\leavevmode\hbox to3em{\hrulefill}\thinspace}
\providecommand{\MR}{\relax\ifhmode\unskip\space\fi MR }
\providecommand{\MRhref}[2]{%
  \href{http://www.ams.org/mathscinet-getitem?mr=#1}{#2}
}
\providecommand{\href}[2]{#2}

\enddocument